\documentclass[12pt, leqno]{amsart}

\usepackage{esint} 
\usepackage{amsmath}
\usepackage{amssymb}
\usepackage{amsfonts}
\usepackage{enumerate}
\usepackage{enumitem}
\usepackage{mathrsfs}

\usepackage{mathtools}
\usepackage{xcolor}

\usepackage{float}
\usepackage{tikz}
\usetikzlibrary{lindenmayersystems}
\usepackage{pgfplots}

\usepackage{comment}

\usepackage{amsmath}
\usepackage{amssymb}
\usepackage{amsthm} 
\usepackage{epsf}
\usepackage{graphicx}
\usepackage{esint} 
\usepackage{dsfont}
\usepackage{hyperref} 
\hypersetup{backref,  pdfpagemode=FullScreen,  colorlinks=true} 
\usepackage[french,english]{babel}
\pgfplotsset{compat=1.18}

\numberwithin{equation}{section}
\newtheorem{thm}[equation]{Theorem} 
\newtheorem{lem}[equation]{Lemma}
\newtheorem{cor}[equation]{Corollary}

\theoremstyle{definition}
\newtheorem{defn}[equation]{Definition}

\newtheorem{rmk}[equation]{Remark}

\numberwithin{equation}{section}

\usepackage{cleveref}
\crefname{equation}{}{}
\crefname{lem}{}{}
\crefrangelabelformat{equation}{(#3#1#4)--(#5#2#6)}

\newcommand\abs[2][empty]{\csname#1\endcsname \lvert{#2}\csname#1\endcsname\rvert}
\newcommand\doublebar[2][empty]{\csname#1\endcsname \lVert{#2}\csname#1\endcsname\rVert}

\DeclareMathOperator*{\osc}{osc}

\newcommand{\B}{\mathbf{B}}
\newcommand\dist{\mathop{\mathrm{dist}}\nolimits}

\newcommand\Div{\mathop{\mathrm{div}}\nolimits}
\newcommand\Tr{\mathop{\mathrm{Tr}}\nolimits}

\newcommand\supp{\mathop{\mathrm{supp}}\nolimits}
\newcommand\diam{\mathop{\mathrm{diam}}\nolimits}

\newcommand{\dimH}{\dim_{\mathrm H}}

\newcommand\R{\mathbb{R}}

\newcommand\N{\mathbb{N}}

\makeatletter\def\HyPsd@CatcodeWarning#1{}\makeatother

\title[]{Elliptic theory and dimension drop for operators with divergence-free form-bounded drifts in one-sided chord-arc domains}

\author{Svitlana Mayboroda, Alberto Pacati}

\begin{document}
\begin{abstract}
We develop a boundary value and potential theory for elliptic operators
\[
L=-\Div(A\nabla)+\B\cdot\nabla
\]
with form-bounded drift in one-sided chord-arc domains. Assuming that the positive part of $\Div\B$ has sufficiently small form bound, we establish local regularity estimates and solvability of the Dirichlet problem, and construct the associated harmonic measure. Under the additional assumption $\Div\B\le0$, we construct a Green's function and prove pointwise and gradient estimates. For the characteristic distributional identity with a Dirac mass at the pole to hold, we will additionally assume either $\Div\B=0$ or $\B\in L^{n,1}_{\textup{loc}}(\Omega)$. Moreover, in the divergence-free case, we obtain a local lower bound for the Green's function, CFMS estimates, doubling of harmonic measure, a comparison principle for nonnegative solutions with zero trace, and a change of pole formula. Finally, we construct, in the upper half-space, an operator of the form $-\Delta+\B\cdot\nabla$ with $\B$ form-bounded and divergence-free, whose harmonic measure exhibits dimension drop.
\end{abstract}

\maketitle
\setcounter{tocdepth}{2}
\tableofcontents

\setcounter{tocdepth}{3}
\clearpage

\section{Introduction}
Our work is divided into two parts; in the first part, we study properties of elliptic equations with singular drifts in rough domains. In the second part, we show that such operators may exhibit dimension drop even in the upper half space.

Let $n\ge 3$, and let $\Omega\subset\R^n$ be a one-sided chord-arc domain with unbounded complement. We develop an elliptic theory for the uniformly elliptic operator with drift
\begin{equation}\label{eqn:operator}
L:=-\Div(A\nabla)+\mathbf{B}\cdot\nabla,
\end{equation}
where $A:\Omega\to\R^{n\times n}$ is a real uniformly elliptic matrix-valued function (see Definition~\ref{defn:ellip}), and $\B$ is a singular drift satisfying a homogeneous form-boundedness assumption, namely
\[
\B\in\mathbf{\dot F^2}_\mu(\Omega).
\]

\begin{defn}\label{defn:fb}
We say that a vector field $\mathbf{B}:\Omega\to\R^n$ is $L^2$-homogeneously form-bounded in $\Omega$ if $\mathbf{B}\in L^2_{\mathrm{loc}}(\Omega)$ and there exists a constant $0<\mu<\infty$ such that
\begin{equation}\label{eqn:fb}
\int_{\Omega} |\mathbf{B}|^2\varphi^2
\le \mu\|\nabla \varphi\|_2^2
\qquad
\forall\,\varphi\in C^\infty_c(\Omega).
\end{equation}
In this case, we write $\mathbf{B}\in \mathbf{\dot F^2}_\mu(\Omega)$, and call $\mu$ a homogeneous form bound for $\mathbf{B}$.
\end{defn}
In \cite{Hof26}, the author developed a robust elliptic theory for $L$ under the assumption that
\begin{equation}\label{eqn:small}
\sup_{X\in\Omega}|\B(X)|\dist(X,\partial\Omega)\le \epsilon
\end{equation}
for some sufficiently small $\epsilon>0$.

By the Hardy inequality in \cite{Anc86}, a drift satisfying~\eqref{eqn:small} also satisfies~\eqref{eqn:fb}, with a form bound controlled by $\epsilon^2$. In contrast with \cite{Hof26}, however, we impose no smallness assumption on the form bound $\mu$ of $\B$. Instead, we require a form-boundedness condition on the positive part of the divergence of $\B$, namely
\[
(\Div\B)_+\in\mathbf{\dot F^1}_\delta(\Omega).
\]

\begin{defn}\label{defn:wfb}
We say that a function $f:\Omega\to\R$ is weakly homogeneously form-bounded in $\Omega$ if $f\in L^1_{\mathrm{loc}}(\Omega)$ and there exists a constant $0<\delta<\infty$ such that
\begin{equation}\label{eqn:wfb}
\int_{\Omega} |f|\varphi^2
\le \delta\|\nabla \varphi\|_2^2
\qquad
\forall\,\varphi\in C^\infty_c(\Omega).
\end{equation}
In this case, we write $f\in \mathbf{\dot F^1}_\delta(\Omega)$, and call $\delta$ a weak homogeneous form bound for $f$.
\end{defn}
\begin{rmk}
Notice that the form bound $\mu$ of $\B$ and the form bound $\delta$ of $(\Div\B)_+$  have different homogeneities with respect to the amplitude of $\B$: under the scaling $\B\mapsto t\B$, they scale as $t^2\mu$ and $t\delta$, respectively. Thus, $\sqrt{\mu}$ and $\delta$ have the same homogeneity.
\end{rmk}

There is an extensive literature on local estimates for solutions of
\[
-\Div(A\nabla u)+\B\cdot\nabla u=0
\]
under assumptions on the drift $\B$ related to those considered here. In \cite{LisZ04}, the authors proved a Gaussian upper bound for the heat kernel associated with the parabolic counterpart of~\eqref{eqn:operator} in the case $A=I$, assuming
\[
\B\in \mathbf{\dot F^2}_\mu(\R^n),
\qquad
\Div\B\in \mathbf{\dot F^1}_\delta(\R^n),
\]
with $\delta<2$, together with an additional smallness condition on the global Kato norm of $(\Div\B)_+$. Under the stronger assumption $\Div\B=0$, they also obtained a Gaussian lower bound for the heat kernel and H\"older continuity of solutions to $Lu=0$.

The case of merely measurable, uniformly elliptic coefficient matrices $A$ has also been studied extensively. Several assumptions on the drift are known to imply elliptic or parabolic Harnack inequalities for nonnegative solutions. We mention a few representative results. In \cite{NazU12}, the drift $\B$ is assumed to belong to a suitable Morrey space, or to an appropriate $L^q$ space, and to satisfy $\Div\B\le0$. In \cite{SerSSZ12}, the authors consider drifts $\B\in BMO^{-1}$, namely divergence-free vector fields given by the divergence of a skew-symmetric matrix with $BMO$ entries. In \cite{Har15}, several assumptions on $\B$ are treated, including
\[
\B\in \mathbf{\dot F^2}_\mu(\Omega),
\qquad
\Div\B\le0,
\]
as well as the condition
\[
\B\in\mathbf{\dot F^2}_\delta(\Omega),
\]
with $\delta$ sufficiently small relative to the ellipticity constant of $A$. 
In \cite{Har16}, the author works with solutions $u\in H^1_0(\Omega)$ to $-\Div(A\nabla u)+\B\cdot\nabla u=F$ in a bounded domain $\Omega$, where 
\[
\B\in \mathbf{\dot F^2}_\mu(\Omega),
\qquad
\Div\B=0,
\]
and $F\in H^{-1}_0(\Omega)$ is a finite Radon measure. He proves global weak Lorentz estimates and two-sided interior potential estimates for $u$. No smallness of the form bound $\mu$ is required. In \cite{KinS22}, the authors proved two-sided Gaussian bounds for the heat kernel of the parabolic operator associated with $L$ under the assumptions
\[
\B\in \mathbf{\dot F^2}_\mu(\R^n),
\qquad
|\Div\B|\in \mathbf{\dot F^1}_\delta(\R^n),
\]
with $\delta$ sufficiently small, together with a smallness condition on $\Div\B$ in a suitable Kato class. Moreover, the spaces $\mathbf{\dot F^2}_\mu(\R^n)$ and $\mathbf{\dot F^1}_\delta(\R^n)$ may be replaced by more general classes of multiplicatively form-bounded functions. Related assumptions on $\B$ and $\Div\B$, in a weaker form, are also used in \cite{KinV} to establish a Harnack inequality for nonnegative solutions of $Lu=0$. In \cite{Sak21}, the author constructs a Green's function for elliptic operators with lower-order terms in Lorentz spaces in general domains. In the special case of operators of the form~\eqref{eqn:operator}, they prove a global Green's function upper bound for $\B\in L^{n,1}(\Omega)$, without imposing any condition on $\Div\B$. In~\cite{Hof26}, the author proved local upper and lower bounds for the Green's function under the smallness assumption~\eqref{eqn:small}. 

In the first part of the present work, we assume that
\[
\B\in \mathbf{\dot F^2}_\mu(\Omega)
\qquad\text{and}\qquad
(\Div\B)_+\in\mathbf{\dot F^1}_\delta(\Omega),
\]
for some $0<\mu<\infty$ and $0<\delta<\lambda$, where $\lambda$ is the ellipticity constant of $A$. Under these assumptions, we develop the local regularity theory for solutions of $Lu=0$, including local boundedness and weak Harnack estimates, and prove H\"older continuity up to the boundary for local solutions with continuous boundary data. We also establish solvability of the continuous Dirichlet problem and construct the associated harmonic measure.

Under the additional assumption $\Div\B\le0$, we construct a Green function and prove a global upper bound, up to a multiplicative constant, by the fundamental solution of the Laplacian in $\R^n$, together with the corresponding gradient estimates. For the Green's function that we construct to satisfy the corresponding distributional identity with a Dirac mass at the pole, we will need to assume either $\Div\B=0$ or $\B\in L^{n,1}_{\textup{loc}}(\Omega)$. This is because of the lack of integrability of the Green's function near its pole. 
 Finally, under the assumption $\Div\B=0$, we prove a local lower bound for the Green function, CFMS-type estimates relating the Green function and harmonic measure (see \cite{CafFMS81}), doubling of harmonic measure, a boundary Harnack principle, and a change-of-pole formula for harmonic measure. Our proof of the change-of-pole formula does not rely on applying the boundary Harnack principle to the adjoint Green function. In particular, it establishes the formula also under the smallness assumption~\eqref{eqn:small}, without imposing any condition on $\Div\B$. We point out that the authors of~\cite{Hof26} explicitly left open whether the change-of-pole formula holds in this setting, even under a smallness assumption.

To the best of our knowledge, the boundary value and potential theory developed here is new under these form-boundedness assumptions, even for Lipschitz domains. Our approach builds on classical techniques developed for uniformly elliptic operators
\[
L_0=-\Div(A\nabla)
\]
on rough domains, but the presence of the singular drift requires a number of nontrivial modifications. We exploit the form-boundedness assumptions throughout and adapt ideas from the works on singular and form-bounded drifts discussed above in order to recover the local regularity, Green function, and harmonic measure estimates needed for the boundary theory.

In the second part of our work, we investigate the dimension of the harmonic measure for operators of the form~\eqref{eqn:operator}.
In \cite{Hof26}, extending his previous joint work with
Lewis \cite{HofL01} on Lipschitz domains (see also
\cite{KenP01}), Hofmann proved an extrapolation result
for the harmonic measure of~\eqref{eqn:operator},
without requiring the smallness assumption~\eqref{eqn:small}.
Instead, the drift $\B$ is assumed to satisfy the
Carleson measure condition
\begin{equation}\label{eqn:carl:cond}
\sup_{\substack{x\in\partial\Omega\\
0<r\le\diam(\Omega)}}
r^{-(n-1)}
\int_{B(x,r)\cap\Omega}
|\B_*(X)|^2\dist(X,\partial\Omega)\,dX
\le C<\infty,
\end{equation}
where $\Omega\subset\R^n$ and
$\B_*(X):=\|\B\|_{L^\infty(
B(X,\dist(X,\partial\Omega)/4))}$.
No smallness assumption is imposed on $C$.

More precisely, under~\eqref{eqn:carl:cond},
the following implication holds:
\begin{equation}\label{eqn:extrapolation}
\omega_{-\Div(A\nabla)}
\in\textup{weak-}A_\infty(\sigma)
\implies
\omega_{-\Div(A\nabla)+\B\cdot\nabla}
\in\textup{weak-}A_\infty(\sigma),
\end{equation}
where $\sigma$ denotes the restriction of the
$(n-1)$-dimensional Hausdorff measure to $\partial\Omega$.
We recall that the property
$\omega\in\textup{weak-}A_\infty(\sigma)$
is equivalent to $\omega\ll\sigma$, with the
corresponding Poisson kernel satisfying a weak
reverse H\"older inequality. In turn, this is
equivalent to the solvability of the
$L^p$-Dirichlet problem for some $1<p<\infty$
(see, e.g., \cite{Hof19}).

The implication~\eqref{eqn:extrapolation} is
established by verifying a criterion originating
from \cite{BL04} (see \cite[Lemma 6.1]{Hof26}).
Importantly, this criterion yields both the
weak-$A_\infty$ property and Bourgain's
estimate~\eqref{eqn:bourgain} for $L$.

Moreover, assuming only that $\B$
satisfies~\eqref{eqn:small} for some $\epsilon>0$,
possibly large, and that Bourgain's estimate holds,
Hofmann establishes a local pointwise upper bound
for the Green's function, the CFMS estimates,
and the doubling property of the harmonic measure.
Consequently, these estimates also hold under
the assumptions~\eqref{eqn:carl:cond} and
$\omega_{-\Div(A\nabla)}
\in\textup{weak-}A_\infty(\sigma)$,
without any smallness requirement on the drift.

In the first part of our work, we use the divergence-free condition to overcome the possible largeness of the drift's form-bound, whereas \cite{Hof26} instead exploits the smallness condition~\eqref{eqn:small}. It is then interesting to investigate what we can say about the relationship between $\omega_L$ and $\sigma$ if we combine these two assumptions. More explicitly, it is natural to ask whether replacing the Carleson hypothesis~\eqref{eqn:carl:cond} with the assumptions
\begin{equation}\label{eqn:count:assumption}
\Div\B=0,\quad |\B|(X)\le\frac{\epsilon}{\dist(X,\partial\Omega)}, \quad\forall X\in\Omega,
\end{equation}
for some fixed small $\epsilon$, still yields~\eqref{eqn:extrapolation}. The answer turns out to be negative even in a flat space. Indeed, for any $\epsilon>0$, we present an example\footnote{The authors acknowledge that such an example was generated using AI, as stressed in our acknowledgment at the end of the introduction.} of a vector field $\B$ satisfying~\eqref{eqn:count:assumption} in $\Omega=\R^3_+$, for which the harmonic measure $\omega_{-\Delta+\B\cdot\nabla}$ is singular with respect to $\sigma$. Moreover, this harmonic measure exhibits a dimension drop, that is
\[
\dimH\omega_{-\Delta+\B\cdot\nabla}<2=\dimH\partial\R^3_+.
\]
\begin{rmk}\label{rmk:2dim}
Our construction is $2$-dimensional; indeed, $\B$ depends only on two variables. However, we prove the dimension drop in three dimensions because, in the first part of the paper, we chose to develop the elliptic theory for the operator $L$ in dimension $n\ge 3$. The theory needed to prove the dimension drop could also be developed in dimension $n=2$ in the full generality of one-sided chord-arc domains. We avoid doing so to prevent our work from becoming too long.  
\end{rmk}
For uniformly elliptic operators in divergence form $L_0=-\Div(A\nabla)$, the first examples of singular harmonic measures were given in \cite{ModM80} and \cite{CafFK81}. The construction in \cite{ModM80} shares some similarities with ours. In their two-dimensional example, the matrix $A$ is given by
\[
A(x,y)=\begin{pmatrix}
1    & 0          \\
0      & \alpha(x,y),    
\end{pmatrix}
\]
where the anisotropy $\alpha(x,y)$ is defined in the thin layers $\{k_{m+1}^{-1}\le y<k_m^{-1}\}$, $k_{m+1}\ge 2k_m$, as an interpolation between horizontal oscillations of small wavelengths $h_{m+1}^{-1}$ and $h_{m}^{-1}$, with $h_{m+1}\ge 4 h_m$. The amplitude of these oscillations approaches $0$ as $y$ approaches $0$, so that $\alpha$ is continuous on the closed upper half plane $\overline{\R^2_+}$. In our example, we don't define the vector field layer by layer. We still introduce fast horizontal oscillations, but we superpose them using a harmonic potential:
\begin{align*}
 &\Phi(x,y,z)
   =\epsilon\sum_{j=0}^{\infty}e^{-q^jz}\cos(q^jx),\quad z>0,\\
 &\B(x,y,z)=-\nabla\Phi(x,y,z),\quad z>0,
\end{align*}
where $q$ is chosen sufficiently large to ensure adequate scale separation in the $z$-direction. This choice allows us to make the drift divergence-free and, in each summand $e^{-q^jz}\cos(q^jx)$, to match the horizontal wavelength with the characteristic height scale.
In the construction in \cite{ModM80}, the horizontal oscillations generate a Riesz product in the harmonic measure, proving singularity but not a dimension drop. 
In our construction, instead, each mode of the vector field $-\B$ points horizontally toward the maxima of the corresponding cosine and away from its minima, while it points vertically toward the boundary near the maxima and away from the boundary near the minima. As a result, $-\B$ induces a preferred direction for the paths of the associated diffusion process, which translates into the prescription of a set with Hausdorff dimension smaller than $2$ where the harmonic measure concentrates.
Such a mechanism shares some similarities with the construction in~\cite{Wu96}. In that paper, the authors show an example of an operator $L$ in non-divergence form that exhibits a dimension drop in the upper half-space $\R^{n+1}_+:=\{(x,z)\in\R^{n+1}:z>0\}$. The operator has the form
\[
L=-\Delta_{x}- f(x,z)\partial_{zz},
\]
where $\Delta_x$ is the Laplacian in the first $n$ variables, and $f$ is a continuous positive function bounded from above and below. This construction resembles ours because $f$ is defined on dyadic vertical boxes partitioning thin layers approaching the boundary as rescaled copies of the same building block $f_0$. Here $f_0$ is a function defined on the vertical box $I=[0,1]^n\times[0,\gamma]$, chosen to be small on the central vertical core of $I$ and large near its lateral boundary. Consequently, the vertical diffusivity is much stronger near the lateral boundaries of the boxes than in their central regions. These regions act as traps for the associated diffusion process, biasing its successive exits toward the peripheral portions of the dyadic cubes. Iterating this mechanism across scales concentrates harmonic measure on a set of strictly smaller Hausdorff dimension, thereby producing the dimension drop.
Therefore, some ideas of our construction were already present in the literature, but the framework and the way in which such ideas were implemented are substantially different from ours. 

At the beginning of the introduction, we already noticed that vector fields $\B$ that satisfy $|\B|\dist(\cdot,\partial\Omega)\in L^\infty(\Omega)$ also satisfy the form-boundedness condition. We conclude the introduction by describing another class of functions satisfying the form-boundedness condition. We introduce a localized Chang--Wilson--Wolff class on $\Omega$, modeled on the global class introduced in \cite{ChaWW85}.

\begin{defn}
Let $\gamma:[0,\infty)\to[1,\infty)$ be an increasing function such that
\[
\int_{1}^\infty \frac{dt}{t\gamma(t)}<\infty.
\]
We say that a function $f:\Omega\to\R$ belongs to the Chang--Wilson--Wolff class $CWW(\Omega)$ if
\[
\sup_{2Q\subset\Omega}
\fint_Q
|f(X)|\,\ell(Q)^2
\gamma\bigl(|f(X)|\ell(Q)^2\bigr)\,dX
=:S<\infty,
\]
where the supremum is taken over all cubes $Q$ such that $2Q\subset\Omega$. Here $2Q$ denotes the cube concentric with $Q$ and having twice its side length, and $\ell(Q)$ denotes the side length of $Q$.
\end{defn}

If $\Omega$ is a one-sided chord-arc domain, then
\[
|\B|^2\in CWW(\Omega)
\quad\Longrightarrow\quad
\B\in \mathbf{\dot F^2}_\mu(\Omega)
\]
for some $\mu<\infty$.

This implication follows by localizing the argument of \cite{ChaWW85}. More precisely, for each cube $Q$ satisfying $2Q\subset\Omega$, the proof of \cite{ChaWW85} can be localized to yield the estimate
\[
\int_{2Q} |\B|^2\varphi_Q^2
\lesssim
\int_{2Q} |\nabla\varphi_Q|^2
\]
for test functions $\varphi_Q\in C_c^\infty(2Q)$, with a constant independent of $Q$. A general function $\varphi\in C_c^\infty(\Omega)$ can then be decomposed by means of a partition of unity subordinate to a Whitney decomposition of $\Omega$. Applying the localized estimate to each piece and summing, the error terms generated by the cutoff functions are controlled by the global Hardy inequality in \cite{Anc86}. This yields~\eqref{eqn:fb} for arbitrary $\varphi\in C_c^\infty(\Omega)$.

Examples of vector fields $\B$ for which $|\B|^2\in CWW(\Omega)$ include, under the appropriate choice of $\gamma$, vector fields satisfying conditions such as
\[
|\B|\in L^n(\Omega),
\qquad
|\B|\in L^{n,\infty}(\Omega),
\]
as well as vector fields in Campanato--Morrey classes of the form
\[
|\B|\in L^{2s}_{\mathrm{loc}}(\Omega)
\quad\text{for some }s>1,
\qquad
\left(\fint_Q|\B|^{2s}\right)^{1/(2s)}
\le C_s\ell(Q)^{-1}
\]
for every cube $Q$ such that $2Q\subset\Omega$.
\section*{Acknowledgments}

The authors thank Cole Jeznach for helpful conversations and ideas
concerning possible mechanisms by which a drift could cause harmonic
measure to fail to be absolutely continuous with respect to surface
measure.

Svitlana Mayboroda and Alberto Pacati are both partially supported
by the Simons Collaboration on Localization of Waves
(grant 563916, SM) and the Simons Initiative on Geometry of Flows.

\section{Outline of the paper}
The paper is organized as follows. In Section~\ref{sec:3}, we introduce the notation and the geometric and functional analytic framework used throughout the paper. In particular, we recall the main properties of one-sided chord-arc domains, introduce the relevant Sobolev spaces, and discuss the form-bounded assumptions on the drift and on its divergence.

In Section~\ref{sec:4} we prove existence and uniqueness of solutions of $-\Div(A\nabla u)+\B\cdot\nabla u=0$.

We then develop, in Section~\ref{sec:5}, the local regularity theory for solutions of the equation. We prove local boundedness estimates, a weak Harnack inequality, boundary estimates for solutions with zero trace, and the corresponding H\"older continuity and Carleson estimates.

In Section~\ref{sec:6}, we establish the maximum principle, construct the harmonic measure, and prove a Bourgain-type non-degeneracy estimate.

In Section~\ref{sec:7}, we construct the Green's function and establish its main properties. These include pointwise upper bounds, gradient estimates, a local lower bound, representation formulas, and the symmetry relation between the Green's functions of $L$ and $L^*$.

In Section~\ref{sec:8}, we prove the CFMS estimates relating the Green's function and harmonic measure. As consequences, we obtain the doubling property of harmonic measure, a comparison principle for nonnegative solutions vanishing on a portion of the boundary, and the change of pole formula. We also observe that the latter argument applies to the small drift setting without requiring any assumption on $\Div\B$.

Finally, in Section~\ref{sec:9}, we construct an example, in $\R^3_+$, of a divergence-free vector field $\B$ satisfiying the smallness assumption~\eqref{eqn:small}, for which the harmonic measure of the operator $-\Delta+\B\cdot\nabla$ exhibits dimension drop.
\section{Notation and preliminaries}\label{sec:3}

Throughout the manuscript, we use the following notation.
\begin{itemize}
    \item For $X\in\R^n$ and $0<r<\infty$, we denote by $B(X,r)$ the Euclidean ball centered at $X$ with radius $r$.

    \item If $\Omega\subset\R^n$ is open, $x\in\partial\Omega$, and $r>0$, we set
    \[
    D(x,r):=B(x,r)\cap\Omega,
    \qquad
    \Delta(x,r):=B(x,r)\cap\partial\Omega.
    \]

    \item If $E\subset\R^n$ is a Borel set, we denote its Lebesgue measure by $|E|$.

    \item We use lowercase letters for points on $\partial\Omega$ and uppercase letters for points in $\Omega$.

    \item For $X\in\Omega$, we set
    \[
    \delta(X):=\dist(X,\partial\Omega).
    \]

    \item The value of a constant denoted by $C$ may change from line to line, but its dependence on the underlying parameters will remain unchanged.
\end{itemize}

We work throughout with real uniformly elliptic matrices.

\begin{defn}\label{defn:ellip}
Let $A:\Omega\to\R^{n\times n}$ be a measurable matrix-valued function. We say that $A$ is uniformly elliptic if there exists $\lambda>0$ such that
\begin{align}
\lambda|v|^2
&\le A(X)v\cdot v,
&&\forall\,X\in\Omega,\ \forall\,v\in\R^n,
\label{ellipticity}\\
|A(X)v\cdot w|
&\le \lambda^{-1}|v||w|,
&&\forall\,X\in\Omega,\ \forall\,v,w\in\R^n.
\label{ellipticity1}
\end{align}
We call $\lambda$ the ellipticity constant of $A$.
\end{defn}

We assume throughout that $\Omega$ is a one-sided chord-arc domain, in the sense of the following definition.

\begin{defn}\label{defn:ocd}
We say that $\Omega$ is a one-sided chord-arc domain if it satisfies the following three conditions.
\begin{enumerate}
    \item[\textup{(i)}] \textup{(Interior corkscrew condition)}
    There exists a constant $0<\varepsilon<1$ such that, for every $z\in\partial\Omega$ and every $0<r\le\diam(\Omega)$, there exists a point
    \[
    Z_r\in\Omega\cap B(z,r)
    \]
    satisfying
    \[
    \delta(Z_r)\ge\varepsilon r.
    \]
    Such a point is called a corkscrew point for $z$ at scale $r$.

    \item[\textup{(ii)}] \textup{(Interior Harnack chain condition)}
    There exists a constant $M>1$ such that, for every $0<r\le\diam(\Omega)$, every $\epsilon>0$, every integer $k\ge1$, and all $X,Y\in\Omega$ satisfying
    \[
    |X-Y|=r,
    \qquad
    \min\{\delta(X),\delta(Y)\}\ge\epsilon,
    \qquad
    2^k\epsilon\ge r,
    \]
    there exist points $X_1,\dots,X_{Mk}\in\Omega$ such that
    \[
    X_1=X,
    \qquad
    X_{Mk}=Y,
    \]
    and
    \[
    X_{i+1}\in B\bigl(X_i,\delta(X_i)/2\bigr),
    \qquad i<Mk.
    \]

    \item[\textup{(iii)}] \textup{(Ahlfors--David regularity)}
    There exists a constant $N\ge1$ such that, for every $z\in\partial\Omega$ and every $0<r\le\diam(\partial\Omega)$,
    \begin{equation}\label{eqn:ahlf}
    N^{-1}r^{n-1}
    \le \sigma(\Delta(z,r))
    \le Nr^{n-1},
    \end{equation}
    where $\sigma:=\mathcal H^{n-1}|_{\partial\Omega}$.
\end{enumerate}
\end{defn}

\begin{defn}
If $\Omega$ is a one-sided chord-arc domain, we refer to the constants $\varepsilon$, $M$, and $N$ in Definition~\ref{defn:ocd} as the geometric constants of $\Omega$.
\end{defn}

\begin{rmk}
As a consequence of the interior corkscrew condition, there exists a constant $C\ge1$, depending only on $n$ and on the geometric constants of $\Omega$, such that
\[
C^{-1}r^n
\le |D(z,r)|
\le Cr^n,
\qquad
z\in\partial\Omega,\quad
0<r<\diam(\Omega).
\]
We use this estimate repeatedly without further mention.
\end{rmk}

Throughout the manuscript, we restrict our attention to one-sided chord-arc domains $\Omega$ that are either bounded or unbounded with unbounded boundary. This is equivalent to requiring that $\Omega$ have unbounded complement, or, equivalently,
\[
\diam(\Omega)=\diam(\partial\Omega).
\]
Most of our local results can also be adapted to unbounded domains with bounded complement. Some global statements, however, may fail in that setting. For instance, in the exterior domain
\[
\Omega=\R^n\setminus B(0,1),
\]
the standard harmonic measure need not be a probability measure.

We next recall the construction of Carleson tents.

\begin{lem}\cite[Section 3]{HofM14}\label{lem:tent}
Let $\Omega\subset\R^n$ be a one-sided chord-arc domain with unbounded complement. There exists a constant $K>1$, depending only on $n$ and on the geometric constants of $\Omega$, such that the following holds.

For every $x\in\partial\Omega$ and every $0<r\le\diam(\Omega)$, there exists a one-sided chord-arc domain
\[
T(x,r)\subset\Omega,
\]
called a Carleson tent, such that
\[
D(x,r)
\subset
D(x,5r/4)
\subset
T(x,r)
\subset
D(x,Kr).
\]
The family may be chosen to be increasing in $r$, in the sense that
\[
r<R
\quad\Longrightarrow\quad
T(x,r)\subset T(x,R).
\]
Moreover, the geometric constants of the Carleson tents depend only on $n$ and on the geometric constants of $\Omega$.
\end{lem}

We now introduce the Sobolev spaces used throughout the paper.

\begin{defn}
Let $U\subset\R^n$ be open. We denote by $W^{1,2}(U)$ the standard Sobolev space
\[
W^{1,2}(U)
:=
\{u\in L^2(U):\nabla u\in L^2(U)\}.
\]
We also define
\[
Y^{1,2}(U)
:=
\{u\in L^{2^*}(U):\nabla u\in L^2(U)\},
\]
where
\[
2^*:=\frac{2n}{n-2}.
\]
We equip $Y^{1,2}(U)$ with the norm
\[
\|u\|_{Y^{1,2}(U)}
:=
\|u\|_{L^{2^*}(U)}
+
\|\nabla u\|_{L^2(U)}.
\]
We denote by $Y^{-1,2}(U)$ the dual space of $Y^{1,2}(U)$.
\end{defn}

We will also need to consider boundary traces of Sobolev functions.

\begin{defn}
Let $U\subset\R^n$ be open, let $u\in W^{1,2}(U)$, let $z\in\partial U$, and let $a\in\R$. We say that $u$ has trace $a$ at $z$ if
\[
\lim_{r\to0}
\fint_{B(z,r)\cap U}
|u(X)-a|\,dX
=0.
\]
In this case, we write
\[
\Tr u(z)=a.
\]
\end{defn}

We refer to \cite{DavFM23B} for the existence and properties of the trace operator $\Tr$ on the Sobolev spaces considered here.

For one-sided chord-arc domains, Sobolev functions with vanishing trace satisfy the following boundary Poincar\'e inequality.

\begin{lem}\cite[Corollary 7.9]{DavFM23B}\label{lem:poinc}
Let $\Omega\subset\R^n$ be a one-sided chord-arc domain with unbounded complement. Let $1<\tau\le 2$,  $x\in\partial\Omega$, and $0<r\le\diam(\Omega)$. If $u\in W^{1,2}(D(x,3r))$
and
$
\Tr u=0$ on $\Delta(x,\tau r)$,
then
\begin{equation}\label{eqn:poinc}
\int_{D(x,r)}u^2
\le
Cr^2
\int_{D(x,\tau r)}|\nabla u|^2,
\end{equation}
where $C$ depends only on $\tau$, $n$, and on the geometric constants of $\Omega$.
\end{lem}

We next define the homogeneous Sobolev space of functions with zero trace.

\begin{defn}
Let $\Omega\subset\R^n$ be a one-sided chord-arc domain with unbounded complement. We define
\[
Y^{1,2}_0(\Omega)
:=
\left\{
u\in L^1_{\mathrm{loc}}(\Omega):
\nabla u\in L^2(\Omega)
\text{ and }
\Tr u=0\text{ on }\partial\Omega
\right\},
\]
equipped with the norm
\[
\|u\|_{Y^{1,2}_0(\Omega)}
:=
\|\nabla u\|_{L^2(\Omega)}.
\]
We denote by $Y^{-1,2}_0(\Omega)$ the dual space of $Y^{1,2}_0(\Omega)$.
\end{defn}

Combining \cite[Theorem 3]{Wal91} with a cutoff argument and the boundary Poincar\'e inequality~\eqref{eqn:poinc}, one obtains the following density result; see also \cite[Lemma 9.18]{DavFM23B}.

\begin{lem}\label{lem:dense}
Let $\Omega\subset\R^n$ be a one-sided chord-arc domain with unbounded complement. Then
\[
\overline{C^\infty_c(\Omega)}^{\,Y^{1,2}_0(\Omega)}
=
Y^{1,2}_0(\Omega).
\]
In particular,
\[
Y^{1,2}_0(\Omega)
=
\left\{
u\in Y^{1,2}(\Omega):
\Tr u=0\text{ on }\partial\Omega
\right\},
\]
and the Sobolev inequality
\[
\|u\|_{L^{2^*}(\Omega)}
\le
C_n\|\nabla u\|_{L^2(\Omega)}
\]
holds for every $u\in Y^{1,2}_0(\Omega)$.
\end{lem}

We conclude this section with a few remarks concerning the form-boundedness conditions~\eqref{eqn:fb} and~\eqref{eqn:wfb}.

\begin{rmk}\label{rmk:fb}
Assume that $\B$ satisfies~\eqref{eqn:fb} and that $f$ satisfies~\eqref{eqn:wfb}. Let $\varphi\in Y^{1,2}_0(\Omega)$. By Lemma~\ref{lem:dense}, there exists a sequence
$
\{\varphi_i\}_{i\in\N}\subset C^\infty_c(\Omega)
$
such that
$
\varphi_i\to\varphi$
in $Y^{1,2}_0(\Omega)
$
and, after passing to a subsequence, almost everywhere in $\Omega$. Therefore, by Fatou's lemma,
\[
\begin{aligned}
\int_{\Omega}|\B|^2\varphi^2
&\le
\liminf_{i\to\infty}
\int_{\Omega}|\B|^2\varphi_i^2\\
&\le
\liminf_{i\to\infty}
\mu\|\nabla\varphi_i\|_{L^2(\Omega)}^2
=
\mu\|\nabla\varphi\|_{L^2(\Omega)}^2.
\end{aligned}
\]
Similarly,
\[
\int_{\Omega}|f|\varphi^2
\le
\delta\|\nabla\varphi\|_{L^2(\Omega)}^2.
\]
Thus the integrals
\[
\int_{\Omega}|\B|^2\varphi^2
\qquad\text{and}\qquad
\int_{\Omega}|f|\varphi^2
\]
are finite for every $\varphi\in Y^{1,2}_0(\Omega)$, and the bounds~\eqref{eqn:fb} and~\eqref{eqn:wfb} extend to all such test functions.
\end{rmk}

Since the condition
\[
(\Div\B)_+\in\mathbf{\dot F^1}_\delta(\Omega)
\]
does not impose any regularity assumption on the negative part of $\Div\B$, we specify the weak meaning that we assign to $(\Div\B)_+$.

\begin{defn}\label{defn:divp}
Let $\B\in L^2_{\mathrm{loc}}(\Omega;\R^n)$. We say that
\[
(\Div\B)_+\in L^1_{\mathrm{loc}}(\Omega)
\]
if there exists a nonnegative function
\[
f\in L^1_{\mathrm{loc}}(\Omega)
\]
such that
\begin{equation}\label{eqn:divp}
-\int_{\Omega}\B\cdot\nabla\varphi
\le
\int_{\Omega}f\varphi
\end{equation}
for every nonnegative $\varphi\in C^\infty_c(\Omega)$. Whenever such a function exists, we fix one such $f$ and, with a slight abuse of notation, define
\[
(\Div\B)_+:=f.
\]
\end{defn}

\begin{rmk}
If $\Div\B$ is represented by a locally integrable function, then the condition in Definition~\ref{defn:divp} implies
\[
\max\{\Div\B,0\}\le(\Div\B)_+
\qquad\text{a.e. in }\Omega.
\]
Thus, in this case, our definition is consistent with the usual notion of the positive part, while allowing for the more general situation in which $\Div\B$ is only defined distributionally.
\end{rmk}

\begin{rmk}
Suppose that
$
\B\in\mathbf{\dot F^2}_\mu(\Omega)$
and$
(\Div\B)_+\in\mathbf{\dot F^1}_\delta(\Omega)$.
Then, by Lemma~\ref{lem:dense}, Remark~\ref{rmk:fb}, an approximation argument, and H\"older's inequality, we have
\begin{equation}\label{eqn:bpfb}
-\int_{\Omega}\B\cdot\nabla\varphi\,\psi
\le
\int_{\Omega}(\Div\B)_+\varphi\psi
+
\int_{\Omega}\B\cdot\nabla\psi\,\varphi,
\qquad
0\le\varphi,\psi\in Y^{1,2}_0(\Omega).
\end{equation}
\end{rmk}

\begin{rmk}\label{rmk:scaling}
The conditions
$
\B\in\mathbf{\dot F^2}_\mu(\Omega)
$ and
$(\Div\B)_+\in\mathbf{\dot F^1}_\delta(\Omega)
$
are invariant under the natural scaling of the equation. Consequently, when proving local De Giorgi--Nash--Moser estimates for subsolutions or supersolutions of
$
-\Div(A\nabla u)+\B\cdot\nabla u=0
$
in a ball of radius $r$, we may reduce to the case $r=1$ by scaling.
\end{rmk}

\section{Existence and uniqueness of solutions}\label{sec:4}

\begin{defn}
Let $A:\Omega\to \R^{n\times n}$ be uniformly elliptic, and let $\B$ be a vector field satisfying $\B\in \mathbf{\dot F^2}_\mu(\Omega)$ and $(\Div \B)_+\in \mathbf{\dot F}_{\delta}(\Omega)$, with $\mu>0$ and $0<\delta<\lambda$, where $\lambda$ is the ellipticity constant of $A$.

We define the continuous bilinear form $\alpha:Y^{1,2}(\Omega)\times Y^{1,2}_0(\Omega)\to \R$ by
\begin{equation}\label{eqn:bil}
\alpha(u,\varphi):= \int_{\Omega}A\nabla u\cdot\nabla\varphi+\int_{\Omega}\B\cdot\nabla u\,\varphi,\quad\forall\, u\in Y^{1,2}(\Omega),\,\varphi\in Y^{1,2}_0(\Omega).
\end{equation}

We also define the continuous bilinear form $\beta:Y^{1,2}(\Omega)\times Y^{1,2}_0(\Omega)\to \R$ by
\begin{equation}\label{eqn:beta}
\beta(u,\varphi):= \int_{\Omega}\B\cdot\nabla u\,\varphi,\quad\forall\, u\in Y^{1,2}(\Omega),\,\varphi\in Y^{1,2}_0(\Omega).
\end{equation}

We define the operator $L=-\Div(A\nabla)+\B\cdot\nabla$ as the bounded linear operator $L:Y^{1,2}(\Omega)\to Y^{-1,2}_0(\Omega)$ given by
\begin{equation}\label{def:L}
\langle Lu,\varphi\rangle:=\alpha(u,\varphi),\quad\forall\, u\in Y^{1,2}(\Omega),\,\varphi\in Y^{1,2}_0(\Omega).
\end{equation}
\end{defn}

\begin{rmk}\label{rmk:bil}
The continuity of the bilinear form $\alpha$ defined in~\eqref{eqn:bil} follows from H\"older's inequality, the uniform ellipticity of $A$, and the continuity of $\beta$. The latter follows from H\"older's inequality, the assumption $\B\in \mathbf{\dot F^2}_\mu(\Omega)$, and Remark~\ref{rmk:fb}. Indeed, for any $u\in Y^{1,2}(\Omega)$ and $\varphi\in Y^{1,2}_0(\Omega)$,
\begin{align*}
|\beta(u,\varphi)|
&\le \|\nabla u\|_{L^2(\Omega)}\||\B|\varphi\|_{L^2(\Omega)}\\
&\le \sqrt{\mu}\|\nabla u\|_{L^2(\Omega)}\|\nabla \varphi\|_{L^2(\Omega)}
\le \sqrt{\mu}\|u\|_{Y^{1,2}(\Omega)}\|\varphi\|_{Y^{1,2}_0(\Omega)}.
\end{align*}
\end{rmk}

\begin{defn}\label{def:sol}
Let $A:\Omega\to \R^{n\times n}$ be uniformly elliptic, and let $\B$ be a vector field satisfying $\B\in \mathbf{\dot F^2}_\mu(\Omega)$ and $(\Div \B)_+\in \mathbf{\dot F}_{\delta}(\Omega)$, with $\mu>0$ and $0<\delta<\lambda$, where $\lambda$ is the ellipticity constant of $A$. Let $L=-\Div(A\nabla)+\B\cdot\nabla$ be the operator defined by~\eqref{def:L}.

Let $u\in Y^{1,2}(\Omega)$, $f\in \textup{Lip}_c(\R^n)$, and $F\in Y^{-1,2}_0(\Omega)$. We say that $u$ is a solution to
\begin{equation}\label{eqn:pd}
\begin{cases}
L u = F & \text{in } \Omega, \\
u = f & \text{on } \partial\Omega,
\end{cases}
\end{equation}
if $\Tr u=f$ on $\partial\Omega$ and $Lu=F$ as elements of $Y^{-1,2}_0(\Omega)$. Equivalently,
\begin{equation}\label{eqn:sol:F}
\int_{\Omega}A\nabla u\cdot\nabla\varphi+\int_{\Omega}\B\cdot\nabla u\,\varphi=\langle F, \varphi \rangle,\quad\forall\,\varphi\in C^\infty_c(\Omega).
\end{equation}
If $F=0$, we simply say that $u$ is a solution.

We say that $u$ is a subsolution if
\begin{equation}\label{eqn:sub}
\int_{\Omega}A\nabla u\cdot\nabla\varphi+\int_{\Omega}\B\cdot\nabla u\,\varphi\le0,\quad\forall\,0\le\varphi\in C^\infty_c(\Omega),
\end{equation}
and we say that $u$ is a supersolution if $-u$ is a subsolution.
\end{defn}

\begin{rmk}\label{rmk:sol}
If $u\in W^{1,2}_{\textup{loc}}(\Omega)$ satisfies~\eqref{eqn:sol:F} for every $\varphi\in C^\infty_c(\Omega)$, we say that $u$ is a local weak solution to $Lu=F$.

If $u\in Y^{1,2}(\Omega)$ is a solution to $Lu=F$, then, by the continuity of the bilinear form $\alpha$ (see Remark~\ref{rmk:bil}) and the density of $C^\infty_c(\Omega)$ in $Y^{1,2}_0(\Omega)$, the identity~\eqref{eqn:sol:F} extends to every $\varphi\in Y^{1,2}_0(\Omega)$. Similarly, if $u$ is a subsolution, the class of admissible test functions in~\eqref{eqn:sub} can be extended to all nonnegative $\varphi\in Y^{1,2}_0(\Omega)$.
\end{rmk}

\begin{lem}\label{lem:existence}
Let $\Omega\subset\R^n$ be a one-sided chord-arc domain with unbounded complement. Let $A:\Omega\to \R^{n\times n}$ be uniformly elliptic, and let $\B$ be a vector field satisfying $\B\in \mathbf{\dot F^2}_\mu(\Omega)$ and $(\Div \B)_+\in \mathbf{\dot F}_{\delta}(\Omega)$, with $\mu>0$ and $0<\delta<\lambda$, where $\lambda$ is the ellipticity constant of $A$. Let $L=-\Div(A\nabla)+\B\cdot\nabla$ be the operator defined by~\eqref{def:L}, and let $F\in Y^{-1,2}_0(\Omega)$. Then there exists a unique $u\in Y^{1,2}_0(\Omega)$ satisfying
\begin{equation}\label{eqn:poisson}
L u = F \quad \text{in } \Omega
\end{equation}
in the sense of Definition~\ref{def:sol}. Moreover,
\begin{equation}\label{eqn:est:ex}
\|u\|_{Y^{1,2}(\Omega)}\le C\|F\|_{Y^{-1,2}_0(\Omega)},
\end{equation}
where $C$ depends only on $\lambda$, $n$, $\mu$, and $\delta$.

Let $f\in \textup{Lip}_c(\R^n)$. Then there exists a unique solution $u_f\in Y^{1,2}(\Omega)$ to the Dirichlet problem
\begin{equation}\label{eqn:dirichlet}
\begin{cases}
L u_f = 0 & \text{in } \Omega, \\
\Tr u_f = f & \text{on } \partial\Omega.
\end{cases}
\end{equation}
\end{lem}

\begin{proof}
Let $\alpha$ be the bilinear form defined in~\eqref{eqn:bil}. As observed in Remark~\ref{rmk:bil}, $\alpha$ is continuous, with a continuity constant depending only on $\lambda$ and $\mu$. We claim that $\alpha$ is coercive on $Y^{1,2}_0(\Omega)\times Y^{1,2}_0(\Omega)$. Indeed, if $u\in Y^{1,2}_0(\Omega)$, then, by the uniform ellipticity of $A$,
\[
\alpha(u,u)\ge\lambda \int_{\Omega}|\nabla u|^2+\int_{\Omega}\B\cdot\nabla u\,u.
\]
Approximating $u$ by functions in $C^\infty_c(\Omega)$ and using the definition of $(\Div\B)_+$, as in~\eqref{eqn:bpfb}, we obtain
\begin{equation*}
\int_{\Omega}\B\cdot\nabla u\,u\ge-\int_{\Omega}\B\cdot\nabla u\,u-\int_\Omega (\Div\B)_+u^2,
\end{equation*}
and therefore
\begin{equation}\label{eqn:div:trick}
\int_{\Omega}\B\cdot\nabla u\,u\ge-\frac{1}{2}\int_\Omega (\Div\B)_+u^2.
\end{equation}
Using the assumption $(\Div\B)_+\in \mathbf{\dot F^1}_\delta(\Omega)$, we deduce that
\[
\alpha(u,u)\ge \lambda \int_{\Omega}|\nabla u|^2-\frac{1}{2}\int_\Omega (\Div\B)_+u^2
\ge \left(\lambda-\frac{\delta}{2}\right)\int_{\Omega}|\nabla u|^2.
\]
Since $\delta<\lambda$, it follows that
\[
\alpha(u,u)\ge \left(\lambda-\frac{\delta}{2}\right)\|u\|_{Y^{1,2}_0(\Omega)}^2,
\]
and hence $\alpha$ is coercive.

The Lax--Milgram theorem therefore yields the existence and uniqueness of a solution $u\in Y^{1,2}_0(\Omega)$ to~\eqref{eqn:poisson}. Moreover,
\[
\|\nabla u\|_{L^2(\Omega)}
\le \left(\lambda-\frac{\delta}{2}\right)^{-1}\|F\|_{Y^{-1,2}_0(\Omega)}.
\]
Combining this estimate with the Sobolev inequality from Lemma~\ref{lem:dense} gives~\eqref{eqn:est:ex}.

Let now $f\in\textup{Lip}_c(\R^n)$. Then $f\in Y^{1,2}(\Omega)$, and hence $Lf\in Y^{-1,2}_0(\Omega)$. Let $w\in Y^{1,2}_0(\Omega)$ be the unique solution to~\eqref{eqn:poisson} with $F=-Lf$, and set $u_f:=w+f$. Then $Lu_f=0$ in $\Omega$ and $\Tr u_f=f$ on $\partial\Omega$. Uniqueness follows from linearity and the uniqueness of solutions to~\eqref{eqn:poisson}.
\end{proof}

\begin{rmk}
The inequality~\eqref{eqn:div:trick}, or suitable variants of it, will be used repeatedly in our proofs to estimate integrals of the form $\int_{\Omega}\B\cdot\nabla u\,u\varphi$, where $0\le\varphi\in C^\infty_c(B(x,r))$ (or $\varphi\in \textup{Lip}_b(\Omega)$), for some $x\in\overline\Omega$ and $r>0$, and $\Tr u=0$ on $\Delta(x,r)$. In what follows, we will use the integration-by-parts argument leading to~\eqref{eqn:div:trick} without repeating the intermediate steps. More precisely, we will directly use
\[
\int_{\Omega}\B\cdot\nabla u\,u\varphi\ge -\frac{1}{2}\int_{\Omega}(\Div\B)_+ u^2\varphi-\frac{1}{2}\int_{\Omega}\B\cdot\nabla\varphi\, u^2.
\]
Taking $\varphi=1$ and assuming that $(\Div\B)_+=0$, the formula becomes
\begin{equation}\label{eqn:beta:0}
\beta(u,u)=\int_{\Omega}\B\cdot\nabla u\,u \ge 0,\quad \textup{if }(\Div\B)_+=0.
\end{equation}
\end{rmk}

\begin{rmk}\label{rmk:adjoint}
Let $A$, $\B$, and $L$ be as in Lemma~\ref{lem:existence}. The adjoint operator $L^*:Y^{1,2}_0(\Omega)\to Y^{-1,2}(\Omega)$ is bounded, and the associated bilinear form $\alpha^*$, defined by $\alpha^*(u,v):=\langle L^*u,v\rangle=\langle Lv,u\rangle$, is bounded on $Y^{1,2}_0(\Omega)\times Y^{1,2}(\Omega)$. Its restriction to $Y^{1,2}_0(\Omega)\times Y^{1,2}_0(\Omega)$ is coercive. Therefore, by the Lax--Milgram theorem, for every $F\in Y^{-1,2}_0(\Omega)$ there exists a unique $u^*\in Y^{1,2}_0(\Omega)$ such that
\[
\alpha^*(u^*,\varphi)=\langle F,\varphi\rangle,\qquad \forall\,\varphi\in Y^{1,2}_0(\Omega).
\]
\end{rmk}

\section{Local boundedness and H\"older continuity up to the boundary}\label{sec:5}

In this section, we establish the De Giorgi--Nash--Moser theory for the operator $L=-\Div(A\nabla)+\B\cdot\nabla$. The interior De Giorgi--Nash--Moser theory for operators of this type has been investigated under assumptions on $\B$ very similar to ours (see \cite{LisZ04}, \cite{NazU12}, \cite{Har15}, \cite{KinS22}, \cite{KinV}, and \cite{Hof26}). In particular, when $\Div\B\le 0$ and $\B\in\mathbf{\dot F^2}_\mu(\Omega)$, the corresponding interior estimates were proven in \cite{Har15}.  However, we have not found proofs under our precise assumptions, which allow $(\Div\B)+$ to be nonzero with a sufficiently small $L^1$ form bound. We therefore also provide the corresponding interior estimates for the operator $L$. Since their proofs rely on the same ideas as their boundary counterparts, we omit them and focus our exposition on estimates up to the boundary, where the drift $\B$ may be more singular and the boundary values of solutions and test functions must be taken into account when manipulating equations and inequalities.

\begin{rmk}
Although our goal is to prove local estimates, in a ball $B$, for local solutions, the inequalities defining the assumptions $\B\in\mathbf{\dot F^2}_\mu(\Omega)$ and $(\Div\B)_+\in \mathbf{\dot F^1}_\delta(\Omega)$ are global. This causes no difficulty. Indeed, we apply these global inequalities to suitable cutoffs of the solutions under consideration, so that the region outside the support of the cutoff does not contribute. This is analogous to the standard proofs of the interior elliptic theory for the operator $L_0=-\Div(A\nabla)$, where global Sobolev inequalities are applied to compactly supported cutoffs in order to derive local estimates (see, e.g., \cite[Chapters 3--4]{HanL00}).
\end{rmk}

We first prove a Caccioppoli inequality for subsolutions, which will be used to establish local boundedness and H\"older continuity of solutions.

\begin{lem}\label{lem:Cacc:int}
Let $\Omega\subset\R^n$ be a one-sided chord-arc domain with unbounded complement. Let $A:\Omega\to \R^{n\times n}$ be uniformly elliptic, and let $\B$ be a vector field satisfying $\B\in \mathbf{\dot F^2}_\mu(\Omega)$ and $(\Div \B)_+\in \mathbf{\dot F}_{\delta}(\Omega)$, with $\mu>0$ and $0<\delta<\lambda$, where $\lambda$ is the ellipticity constant of $A$.
Let $X\in\Omega$, $0<4r\le\delta(X)$, and let $u\in W^{1,2}(B(X,4r))$ be a subsolution in $B(X,4r)$, that is,
\begin{equation}\label{eqn:subsol:int}
\int_{B(X,4r)}A\nabla u\cdot\nabla\varphi+\int_{B(X,4r)}\B\cdot\nabla u\,\varphi\le0
\end{equation}
for every nonnegative $\varphi\in C^\infty_c(B(X,4r))$.

Set $v:=(u-k)_+$, with $k\ge0$. Then, for every nonnegative $\eta\in C^\infty_c(B(X,4r))$, we have
\begin{align}
\label{c_ineq:int}
\int_{B(X,4r)} |\nabla v|^2 \eta^2 \leq C \int_{B(X,4r)} v^2|\nabla\eta|^2,
\end{align}
where $C$ depends only on $\mu$, $\delta$, and $\lambda$.
\end{lem}

\begin{lem}\label{lem:Cacc}
Let $\Omega\subset\R^n$ be a one-sided chord-arc domain with unbounded complement. Let $A:\Omega\to \R^{n\times n}$ be uniformly elliptic, and let $\B$ be a vector field satisfying $\B\in \mathbf{\dot F^2}_\mu(\Omega)$ and $(\Div \B)_+\in \mathbf{\dot F}_{\delta}(\Omega)$, with $\mu>0$ and $0<\delta<\lambda$, where $\lambda$ is the ellipticity constant of $A$.
Let $x\in\partial\Omega$, $0<5r\le\diam(\Omega)$, and let $u\in W^{1,2}(D(x,5r))$ be a subsolution with zero trace in $B(x,4r)$, that is, $\Tr u=0$ on $\Delta(x,4r)$ and
\begin{equation}\label{eqn:subsol}
\int_{D(x,4r)}A\nabla u\cdot\nabla\varphi+\int_{D(x,4r)}\B\cdot\nabla u\,\varphi\le0
\end{equation}
for every $\varphi\in Y^{1,2}(\Omega)$ such that $\varphi\ge0$ in $D(x,4r)$, $\Tr\varphi=0$ on $\Delta(x,4r)$, and $\varphi\equiv0$ on $\Omega\setminus B(x,\rho)$ for some $0<\rho<4r$.

Set $v:=(u-k)_+$, with $k\ge0$. Then, for every nonnegative $\eta\in C^\infty_c(B(x,4r))$, we have
\begin{align}
\label{c_ineq}
\int_{D(x,4r)} |\nabla v|^2 \eta^2 \leq C \int_{D(x,4r)} v^2|\nabla\eta|^2,
\end{align}
where $C$ depends only on $\mu$, $\delta$, and $\lambda$.
\end{lem}

\begin{proof}
Since $\Tr u=0$ on $\Delta(x,4r)$, we also have $\Tr v=0$ on $\Delta(x,4r)$. Thus, we may use $\varphi=v\eta^2$ as a test function in~\eqref{eqn:subsol}. Using that $\nabla u=\nabla v$ on $\supp v$ and the uniform ellipticity of $A$, we obtain
\begin{align*}
\lambda \int_{D(x,4r)}|\nabla v|^2\eta^2
&\le \int_{D(x,4r)}A\nabla u\cdot\nabla v\,\eta^2\\
&= \int_{D(x,4r)}A\nabla u\cdot \nabla(v\eta^2)
   -2\int_{D(x,4r)}A\nabla u\cdot\nabla\eta\,\eta v\\
&\le -\int_{D(x,4r)}\B\cdot\nabla v\,v\eta^2
   -2\int_{D(x,4r)}A\nabla v\cdot\nabla\eta\,\eta v\\
&\le \frac{1}{2}\int_{D(x,4r)}(\Div\B)_+v^2\eta^2
   +\int_{D(x,4r)}\B\cdot\nabla\eta\,v^2\eta
   -2\int_{D(x,4r)}A\nabla v\cdot\nabla\eta\,\eta v.
\end{align*}
Using Young's inequality $ab\le\epsilon a^2+(4\epsilon)^{-1}b^2$ and the uniform ellipticity of $A$, we obtain
\begin{align*}
(\lambda-\epsilon)\int_{D(x,4r)}|\nabla v|^2\eta^2
&\le \frac{1}{2}\int_{D(x,4r)}(\Div\B)_+v^2\eta^2
+\epsilon\int_{D(x,4r)}|\B|^2v^2\eta^2
+C\int_{D(x,4r)}v^2|\nabla\eta|^2,
\end{align*}
where $\epsilon>0$ will be chosen momentarily and $C$ depends only on $\lambda$ and $\epsilon$.

Using the assumptions $\B\in \mathbf{\dot F^2}_\mu(\Omega)$ and $(\Div\B)_+\in \mathbf{\dot F}_{\delta}(\Omega)$, applied to $v\eta$, and choosing $\epsilon>0$ so that $\lambda-\epsilon-2\epsilon\mu-\delta>0$, we obtain
\begin{align*}
\int_{D(x,4r)}|\nabla v|^2\eta^2
\le C\int_{D(x,4r)}v^2|\nabla\eta|^2.
\end{align*}
\end{proof}

Next, we prove local boundedness of the positive part of a subsolution.

\begin{lem}\label{lem:boundedness:int}
Let $\Omega\subset\R^n$ be a one-sided chord-arc domain with unbounded complement. Let $A:\Omega\to \R^{n\times n}$ be uniformly elliptic, and let $\B$ be a vector field satisfying $\B\in \mathbf{\dot F^2}_\mu(\Omega)$ and $(\Div \B)_+\in \mathbf{\dot F}_{\delta}(\Omega)$, with $\mu>0$ and $0<\delta<\lambda$, where $\lambda$ is the ellipticity constant of $A$.
Let $X\in\Omega$, $0<4r\le\delta(X)$, and let $u\in W^{1,2}(B(X,4r))$ be a subsolution in $B(X,4r)$, in the sense of~\eqref{eqn:subsol:int}.

Then, for every $0<p<\infty$,
\begin{align}\label{eqn:sup:int}
\sup_{B(X,r)} u_+ \leq C\biggl( \fint_{B(X,2r)} u_+^p\biggr)^{1/p},
\end{align}
where $C$ depends only on $\mu$, $\delta$, $\lambda$, $p$, and $n$.
\end{lem}

\begin{rmk}
The proof of the case $p=2$ in~\eqref{eqn:sup:int} is identical to that of the corresponding case in~\eqref{eqn:sup}, which is given below. The case $p>2$ follows immediately from H\"older's inequality, whereas the case $0<p<2$ follows by iterating a scaled version of the $p=2$ estimate. This argument is standard (see \cite[Theorem 4.1]{HanL00}) and will also be used, with the appropriate modifications, in the proof of Lemma~\ref{lem:boundedness}. We therefore omit the details here.
\end{rmk}

\begin{lem}\label{lem:boundedness}
Let $\Omega\subset\R^n$ be a one-sided chord-arc domain with unbounded complement.
Let $A:\Omega\to \R^{n\times n}$ be uniformly elliptic, and let $\B$ be a vector field satisfying $\B\in \mathbf{\dot F^2}_\mu(\Omega)$ and $(\Div \B)_+\in \mathbf{\dot F}_{\delta}(\Omega)$, with $\mu>0$ and $0<\delta<\lambda$, where $\lambda$ is the ellipticity constant of $A$.
Let $x\in\partial\Omega$, $0<5r\le\diam(\Omega)$, and let $u\in W^{1,2}(D(x,5r))$ be a subsolution with zero trace on $\Delta(x,4r)$, in the sense of~\eqref{eqn:subsol}.
Then, for every $0<p<\infty$,
\begin{align}\label{eqn:sup}
\sup_{D(x,r)} u_+ \leq C\biggl( \fint_{D(x,2r)} u_+^p\biggr)^{1/p},
\end{align}
where $C$ depends only on $\mu$, $\delta$, $\lambda$, $p$, $n$, and the geometric constants of $\Omega$.
\end{lem}

\begin{proof}
We follow the classical proof of De Giorgi in \cite{DeG57}, which we include for completeness.

By scaling, we may assume that $r=1$ (see Remark~\ref{rmk:scaling}).
For $k>0$ and a decreasing sequence of radii $(s_i)_{i\in\N}$ with $1<s_{i+1}<s_i\le2$, set $A(k,s_i):=\{Y\in D(x,s_i):u(Y)>k\}$. Let $\eta\in C^\infty_c(B(x,4))$ be a nonnegative cutoff function such that $\eta\equiv1$ on $B(x,s_{i+1})$, $\eta\equiv0$ outside $B(x,s_i)$, and $|\nabla\eta|\le C(s_i-s_{i+1})^{-1}$. An application of Lemma~\ref{lem:Cacc} to $v:=(u-k)_+$ gives
\[
\int_{\Omega}|\nabla(v\eta)|^2
=\int_{D(x,s_i)}|\nabla(v\eta)|^2
\le C\int_{D(x,s_i)}v^2|\nabla\eta|^2
\le C(s_i-s_{i+1})^{-2}\int_{D(x,s_i)}v^2.
\]
Set $2^*=\frac{2n}{n-2}$ and $u_k:=u-k$. By the Sobolev embedding and the density of smooth compactly supported functions in $Y^{1,2}_0(\Omega)$ (see Lemma~\ref{lem:dense}), we obtain
\[
\biggl(\int_{A(k,s_{i+1})}u_k^{2^*}\biggr)^{\frac{2}{2^*}}
\le \biggl(\int_{\Omega}(v\eta)^{2^*}\biggr)^{\frac{2}{2^*}}
\le C\int_{\Omega}|\nabla(v\eta)|^2
\le C(s_i-s_{i+1})^{-2}\int_{A(k,s_i)}u_k^2,
\]
which in turn implies
\begin{equation}\label{eqn:first:step}
\int_{A(k,s_{i+1})}u_k^2
\le
C|A(k,s_{i+1})|^{1-\frac{2}{2^*}}
(s_i-s_{i+1})^{-2}
\int_{A(k,s_i)}u_k^2.
\end{equation}

Let $k>h$. On $A(k,s_{i+1})$, we have $u_h>k-h$ and $u_k\le u_h$. Using Markov inequality, we estimate
\[
|A(k,s_{i+1})|
\le (k-h)^{-2}\int_{A(h,s_i)}u_h^2.
\]
Plugging this estimate into the right-hand side of~\eqref{eqn:first:step}, we obtain
\begin{equation}\label{eqn:iteration}
\int_{A(k,s_{i+1})}u_k^2
\le
C(k-h)^{-2\gamma}(s_i-s_{i+1})^{-2}
\biggl(\int_{A(h,s_i)}u_h^2\biggr)^{1+\gamma},
\end{equation}
where $\gamma:=1-\frac{2}{2^*}>0$.

We choose $s_i:=1+2^{-i}$ and $k_i:=M(1-2^{-i})$, where $M>0$ will be chosen below. Setting
\[
Y_i:=\|(u-k_i)_+\|_{L^2(D(x,s_i))},
\]
and taking square roots in~\eqref{eqn:iteration}, we obtain
\[
Y_{i+1}
\le
CM^{-\gamma}2^{(1+\gamma)i}Y_i^{1+\gamma}.
\]
Notice that $k_0=0$ and $s_0=2$, and hence
\[
Y_0=\|u_+\|_{L^2(D(x,2))}.
\]
Setting $Z_i:=Y_i/Y_0$, we have $Z_0=1$ and
\[
Z_{i+1}
\le
C\left(\frac{Y_0}{M}\right)^\gamma
2^{(1+\gamma)i}Z_i^{1+\gamma}.
\]
Choosing
\[
M=C_\gamma Y_0
=C_\gamma\|u_+\|_{L^2(D(x,2))},
\]
with $C_\gamma$ sufficiently large, the standard nonlinear iteration yields $Y_i\to0$ as $i\to\infty$. Since $s_i\to1$ and $k_i\to M$, we conclude that
\[
\|(u-M)_+\|_{L^2(D(x,1))}=0,
\]
that is, $u\le M$ almost everywhere in $D(x,1)$. This proves the case $p=2$ of~\eqref{eqn:sup}.
The case $p>2$ follows immediately from H\"older's inequality. 

We now turn to the case $0<p<2$.
First, we prove a rescaled version of the $L^2$ estimate. Let $r<R<2r$, and let $X\in D(x,r)$ be a Lebesgue point of $u_+$. If $\delta(X)\ge (R-r)/4$, we apply~\eqref{eqn:sup:int} in $B(X,(R-r)/16)$. If $\delta(X)<(R-r)/4$, choose $z\in\partial\Omega$ such that $|X-z|=\delta(X)$. Then $z\in\Delta(x,r+(R-r)/4)$, and we apply the $p=2$ case of~\eqref{eqn:sup} in $D(z,(R-r)/4)$. In either case, we obtain
\begin{equation*}
u_+(X)\le C(R-r)^{-n/2}\|u_+\|_{L^2(D(x,R))}.
\end{equation*}
Taking the supremum over $X\in D(x,r)$ and using H\"older's inequality, we obtain
\begin{align*}
\sup_{D(x,r)}u_+
&\le C(R-r)^{-n/2}
\|u_+\|_{L^p(D(x,R))}^{p/2}
\sup_{D(x,R)}u_+^{(2-p)/2}.
\end{align*}
By Young's inequality, for every $0<\epsilon<1$,
\begin{align*}
\sup_{D(x,r)}u_+
&\le
C_\epsilon(R-r)^{-n/p}\|u_+\|_{L^p(D(x,R))}
+\epsilon\sup_{D(x,R)}u_+.
\end{align*}
Iterating this inequality along a suitable sequence of radii converging to $2r$ yields~\eqref{eqn:sup}; see \cite[Lemma 4.3]{HanL00} for details of the iteration procedure.
\end{proof}

We now prove a weak Harnack inequality for interior supersolutions.

\begin{lem}\label{lem:w:h}
Let $\Omega\subset\R^n$ be a one-sided chord-arc domain with unbounded complement.
Let $A:\Omega\to \R^{n\times n}$ be uniformly elliptic, and let $\B$ be a vector field satisfying $\B\in \mathbf{\dot F^2}_\mu(\Omega)$ and $(\Div \B)_+\in \mathbf{\dot F}_{\delta}(\Omega)$, with $\mu>0$ and $0<\delta<\lambda$, where $\lambda$ is the ellipticity constant of $A$.
Let $X\in\Omega$, $0<4r\le\delta(X)$, and let $u\in W^{1,2}(B(X,4r))$ be a nonnegative supersolution in $B(X,4r)$, that is,
\begin{equation}\label{eqn:supersol:int}
\int_{B(X,4r)}A\nabla u\cdot\nabla\varphi+\int_{B(X,4r)}\B\cdot\nabla u\,\varphi\ge0
\end{equation}
for every nonnegative $\varphi\in C^\infty_c(B(X,4r))$.

Then there exists $p_0>0$, depending only on $\mu$, $\delta$, $\lambda$, and $n$, such that
\begin{equation}\label{eqn:w:h}
\inf_{B(X,r)}u\ge C^{-1}\Bigl(\fint_{B(X,2r)}u^{p_0}\Bigr)^{1/p_0},
\end{equation}
where $C$ depends only on $\mu$, $\delta$, $\lambda$, and $n$.
\end{lem}

\begin{proof}
Up to replacing $u$ with $u+\epsilon$ and letting $\epsilon\to0$ at the end of the proof, we may assume that $u\ge\epsilon>0$.

Let $0\le\psi\in C^\infty_c(B(X,4r))$. Using $\varphi=\psi u^{-2}$ as a test function in~\eqref{eqn:supersol:int}, we obtain
\[
-\int_{B(X,4r)}A\nabla u^{-1}\cdot\nabla\psi
-\int_{B(X,4r)}\B\cdot\nabla u^{-1}\,\psi
\ge
2\int_{B(X,4r)}A\nabla u\cdot\nabla u\,u^{-3}\psi
\ge0,
\]
where in the last inequality we used the ellipticity of $A$. Thus, $u^{-1}$ is a nonnegative subsolution in $B(X,4r)$. By~\eqref{eqn:sup:int}, for every $p>0$,
\[
\inf_{B(X,r)}u
\ge
C^{-1}\Bigl(\fint_{B(X,2r)}u^{-p}\Bigr)^{-1/p}.
\]
Set $w:=\log u-\fint_{B(X,2r)}\log u$. Then
\[
\Bigl(\fint_{B(X,2r)}u^{-p}\Bigr)^{-1/p}
\ge
\Bigl(\fint_{B(X,2r)}e^{p|w|}\Bigr)^{-2/p}
\Bigl(\fint_{B(X,2r)}u^p\Bigr)^{1/p}.
\]
Therefore, it is enough to prove that there exists $p_0>0$ such that
$\fint_{B(X,2r)}e^{p_0|w|}\le C$. By the John--Nirenberg lemma (see \cite[Theorem 3.5]{HanL00}), it suffices to prove that $\log u$ has uniformly bounded mean oscillation in $B(X,2r)$.

Let $B(Z,s)\subset B(X,2r)$. Then $B(Z,2s)\subset B(X,4r)$. By the Poincar\'e inequality,
\[
\fint_{B(Z,s)}
\left|\log u-\fint_{B(Z,s)}\log u\right|
\le
Cs\Bigl(\fint_{B(Z,s)}|\nabla\log u|^2\Bigr)^{1/2}.
\]
Thus, it remains to prove
\begin{equation}\label{eqn:step:int}
\int_{B(Z,s)}|\nabla\log u|^2\le Cs^{n-2}.
\end{equation}
Set $B:=B(Z,s)$ and $B':=B(Z,3s/2)$. Let $0\le\eta\le1$ be a smooth function such that $\eta\equiv1$ on $B$, $\eta\equiv0$ outside $B'$, and $|\nabla\eta|\le Cs^{-1}$. Using $\eta^2u^{-1}$ as a test function in~\eqref{eqn:supersol:int}, we obtain
\begin{align*}
\int_{B'}|\nabla\log u|^2\eta^2
&\le C\int_{B'}A\nabla u\cdot\nabla u\,\eta^2u^{-2}\\
&\le C\int_{B'}\B\cdot\nabla\log u\,\eta^2
+2C\int_{B'}A\nabla\log u\cdot\nabla\eta\,\eta.
\end{align*}
Using Young's inequality, the uniform boundedness of $A$, and the assumption $\B\in\mathbf{\dot F^2}_\mu(\Omega)$, we obtain
\[
\int_{B'}|\nabla\log u|^2\eta^2
\le
C\int_{B'}|\B|^2\eta^2
+C\int_{B'}|\nabla\eta|^2
\le
C\int_{B'}|\nabla\eta|^2
\le
Cs^{n-2}.
\]
Since $\eta\equiv1$ on $B$, this proves~\eqref{eqn:step:int}, ending the proof of~\eqref{eqn:w:h}.
\end{proof}

We obtain the Harnack inequality for nonnegative solutions as an immediate consequence of Lemmas~\ref{lem:w:h} and~\ref{lem:boundedness:int}.

\begin{lem}\label{lem:h}
Let $\Omega\subset\R^n$ be a one-sided chord-arc domain with unbounded complement.
Let $A:\Omega\to \R^{n\times n}$ be uniformly elliptic, and let $\B$ be a vector field satisfying $\B\in \mathbf{\dot F^2}_\mu(\Omega)$ and $(\Div \B)_+\in \mathbf{\dot F}_{\delta}(\Omega)$, with $\mu>0$ and $0<\delta<\lambda$, where $\lambda$ is the ellipticity constant of $A$.
Let $X\in\Omega$, $0<4r\le\delta(X)$, and let $u\in W^{1,2}(B(X,4r))$ be a nonnegative solution in $B(X,4r)$, that is,
\begin{equation}\label{eqn:sol:int}
\int_{B(X,4r)}A\nabla u\cdot\nabla\varphi+\int_{B(X,4r)}\B\cdot\nabla u\,\varphi=0
\end{equation}
for every $\varphi\in C^\infty_c(B(X,4r))$.

Then,
\begin{equation}\label{eqn:h}
\sup_{B(X,r)}u\le C\inf_{B(X,r)}u,
\end{equation}
where $C$ depends only on $\mu$, $\delta$, $\lambda$, and $n$.
\end{lem}

Since Lemma~\ref{lem:boundedness:int} ensures that local solutions are locally bounded, a standard consequence of the Harnack inequality~\eqref{eqn:h} is the H\"older continuity of solutions.

\begin{lem}\label{lem:h:int}
Let $\Omega\subset\R^n$ be a one-sided chord-arc domain with unbounded complement.
Let $A:\Omega\to \R^{n\times n}$ be uniformly elliptic, and let $\B$ be a vector field satisfying $\B\in \mathbf{\dot F^2}_\mu(\Omega)$ and $(\Div \B)_+\in \mathbf{\dot F}_{\delta}(\Omega)$, with $\mu>0$ and $0<\delta<\lambda$, where $\lambda$ is the ellipticity constant of $A$.
Let $X\in\Omega$, $0<4r\le\delta(X)$, and let $u\in W^{1,2}(B(X,4r))$ be a solution in $B(X,4r)$, in the sense of~\eqref{eqn:sol:int}.

Then there exists $\alpha>0$ such that, for every $0<s<r$,
\begin{equation}\label{eqn:höld:int}
\osc_{B(X,s)}u\le C\biggl(\frac{s}{r}\biggr)^\alpha\biggl(\fint_{B(X,2r)}u^2\biggr)^{1/2},
\end{equation}
where $C$ and $\alpha$ depend only on $\mu$, $\delta$, $\lambda$, and $n$.
\end{lem}

We now turn to the H\"older continuity up to the boundary of solutions with H\"older continuous boundary data. The key step is the following density lemma for supersolutions, inspired by \cite[Lemma 8.14]{DavFM21}.

\begin{lem}\label{lem:density}
Let $\Omega\subset\R^n$ be a one-sided chord-arc domain with unbounded complement.
Let $A:\Omega\to \R^{n\times n}$ be uniformly elliptic, and let $\B$ be a vector field satisfying $\B\in \mathbf{\dot F^2}_\mu(\Omega)$ and $(\Div \B)_+\in \mathbf{\dot F}_{\delta}(\Omega)$, with $\mu>0$ and $0<\delta<\lambda$, where $\lambda$ is the ellipticity constant of $A$.
Let $x\in\partial\Omega$, $0<5r\le\diam(\Omega)$, and let $u\in W^{1,2}(D(x,5r))$ be a nonnegative supersolution in $B(x,4r)$, that is,
\begin{equation}\label{eqn:supersol}
\int_{D(x,4r)}A\nabla u\cdot\nabla\varphi+\int_{D(x,4r)}\B\cdot\nabla u\,\varphi\ge0
\end{equation}
for every $\varphi\in Y^{1,2}(\Omega)$ such that $\varphi\ge 0$ in $D(x,4r)$, $\Tr \varphi=0$ on $\Delta(x,4r)$, and $\varphi\equiv 0$ on $\Omega\setminus B(x,\rho)$ for some $0<\rho<4r$.

Assume that $\Tr u\equiv 1$ on $\Delta(x,4r)$. Then
\begin{equation}\label{eqn:density}
\inf_{D(x,r)}u\ge c,
\end{equation}
where $c>0$ depends only on $\mu$, $\delta$, $\lambda$, $n$, and the geometric constants of $\Omega$.
\end{lem}

\begin{proof}
By scaling, we may assume that $r=1$.

Let $\epsilon>0$ be a small number that will be sent to zero at the end of the proof. Set $u_\epsilon:=\min(1,u+\epsilon)$, and notice that $\Tr u_\epsilon=1$ on $\Delta(x,4)$. Let $g_\epsilon:\R\to\R$ be a smooth Lipschitz function such that $g_\epsilon(s)=\ln(s)$ for every $\epsilon\le s\le 1$, and define $v_\epsilon:=-g_\epsilon(u_\epsilon)$. Then $v_\epsilon=-\ln u_\epsilon$ in $D(x,4)$ and $\Tr v_\epsilon\equiv 0$ on $\Delta(x,4)$. We first prove that $v_\epsilon$ is a subsolution in $B(x,4)$. To this end, let $\varphi\in Y^{1,2}(\Omega)$ satisfy $\Tr\varphi\equiv 0$ on $\Delta(x,4)$, $\varphi\ge0$ in $D(x,4)$, and $\varphi\equiv0$ on $\Omega\setminus B(x,\rho)$ for some $0<\rho<4$. We compute
\begin{align*}
\int_{D(x,4)}A\nabla v_\epsilon\cdot\nabla\varphi+\int_{D(x,4)}\B\cdot\nabla v_\epsilon\,\varphi
=&-\int_{D(x,4)} A\nabla u_\epsilon\cdot\nabla\biggl(\frac{\varphi}{u_\epsilon}\biggr)
-\int_{D(x,4)}\B\cdot\nabla u_\epsilon\frac{\varphi}{u_\epsilon}\\
&-\int_{D(x,4)}A\nabla u_\epsilon\cdot\nabla u_\epsilon\,\varphi u_\epsilon^{-2}.
\end{align*}
Since $A\nabla u_\epsilon\cdot\nabla u_\epsilon\ge0$ by ellipticity, in order to prove that the right-hand side is nonpositive, it is enough to prove that $u_\epsilon$ is a supersolution in $D(x,4)$, that is,
\begin{equation}\label{eqn:minsol}
-\int_{D(x,4)} A\nabla u_\epsilon\cdot\nabla\psi-\int_{D(x,4)}\B\cdot\nabla u_\epsilon\,\psi\le 0,
\end{equation}
for every $\psi\in Y^{1,2}(\Omega)$ such that $\Tr\psi\equiv0$ on $\Delta(x,4)$, $\psi\ge0$ in $D(x,4)$, and $\psi\equiv0$ on $\Omega\setminus B(x,\rho)$ for some $0<\rho<4$.

Fix such a $\psi$, and let $(\Phi_i)_{i\in\N}$ be a sequence of smooth functions $\Phi_i:\R\to\R$ such that $0\le\Phi_i,\Phi_i'\le2$, $\Phi_i(s)\equiv1$ for $s\ge1$, and $\Phi_i$ is concave and nondecreasing on $[0,\infty)$. Moreover, as $i\to\infty$, $\Phi_i\to\min(1,\cdot)$ uniformly on $[0,\infty)$ and $\Phi_i'(s)\to\mathds{1}_{[0,1]}(s)$ for every $s\in[0,\infty)$, where $\mathds{1}_S$ denotes the indicator function of a set $S$. In particular, $\Phi_i(u+\epsilon)$ converges to $u_\epsilon$ in $W^{1,2}(D(x,4))$. Thus, to prove~\eqref{eqn:minsol}, it suffices to show, for every $i\in\N$, that
\[
-\int_{D(x,4)} A\nabla \Phi_i(u+\epsilon)\cdot\nabla\psi-\int_{D(x,4)}\B\cdot\nabla \Phi_i(u+\epsilon)\,\psi\le 0.
\]
We write the left-hand side as
\begin{align*}
-\int_{D(x,4)} A\nabla \Phi_i(u+\epsilon)\cdot\nabla\psi
-\int_{D(x,4)}\B\cdot\nabla \Phi_i(u+\epsilon)\,\psi
=&\int_{D(x,4)} A\nabla u\cdot\nabla u\,\psi\Phi_i''(u+\epsilon)\\
&-\int_{D(x,4)} A\nabla u\cdot\nabla(\psi\Phi_i'(u+\epsilon))\\
&-\int_{D(x,4)}\B\cdot\nabla u\,\psi\Phi'_i(u+\epsilon).
\end{align*}
Since $\Phi_i$ is concave and $A$ is uniformly elliptic, the first term on the right-hand side is nonpositive. Since $\Phi_i$ is nondecreasing, we may use $\psi\Phi_i'(u+\epsilon)$ as a test function in~\eqref{eqn:supersol}, obtaining that the sum of the second and third terms on the right-hand side is nonpositive. This proves~\eqref{eqn:minsol}.

Therefore, $v_\epsilon$ is a nonnegative subsolution with zero trace on $\Delta(x,4)$. By~\eqref{eqn:sup} and Lemma~\ref{lem:poinc},
\begin{equation*}
\sup_{D(x,1)} v_\epsilon
\le C\biggl(\int_{D(x,2)}v_\epsilon^2\biggr)^{1/2}
\le C\biggl(\int_{D(x,3)}|\nabla v_\epsilon|^2\biggr)^{1/2}.
\end{equation*}
This can be rewritten as
\begin{equation}\label{eqn:inf:first}
\min(1,u+\epsilon)\ge \exp\bigl(-C\|\nabla v_\epsilon\|_{L^2(D(x,3))}\bigr)
\quad\text{in }D(x,1).
\end{equation}
Therefore, it remains to prove that
\begin{equation}\label{eqn:grad:bound}
\|\nabla v_\epsilon\|_{L^2(D(x,3))}\le C,
\end{equation}
with $C$ independent of $\epsilon$. Indeed,~\eqref{eqn:grad:bound} and~\eqref{eqn:inf:first}, followed by letting $\epsilon\to0$, yield~\eqref{eqn:density}.

 Let $\eta\in C^\infty_c(B(x,4))$ be a nonnegative function such that $\eta\equiv1$ on $B(x,3)$ and $|\nabla\eta|\le C$. Using $\eta^2(u_\epsilon^{-1}-1)$ as a test function in~\eqref{eqn:minsol} gives
\begin{align*}
\lambda \int_{D(x,4)}|\nabla v_\epsilon|^2\eta^2
&\le \int_{D(x,4)}A\nabla u_\epsilon\cdot\nabla u_\epsilon\,\eta^2u_\epsilon^{-2}\\
&=-\int_{D(x,4)}A\nabla u_\epsilon\cdot\nabla[(u_\epsilon^{-1}-1)\eta^2]
+2\int_{D(x,4)}A\nabla u_\epsilon\cdot\nabla\eta\,\eta(u_\epsilon^{-1}-1)\\
&\le -\int_{D(x,4)}\B\cdot\nabla v_\epsilon(1-u_\epsilon)\eta^2
-2\int_{D(x,4)}A\nabla v_\epsilon\cdot\nabla\eta\,\eta(1-u_\epsilon).
\end{align*}
Applying Young's inequality to the two terms on the right-hand side and absorbing the resulting terms involving $|\nabla v_\epsilon|^2\eta^2$ into the left-hand side, we obtain
\[
\int_{D(x,3)}|\nabla v_\epsilon|^2
\le \int_{D(x,4)}|\nabla v_\epsilon|^2\eta^2
\le C\int_{D(x,4)}|\B|^2(1-u_\epsilon)^2\eta^2
+C\int_{D(x,4)}|\nabla\eta|^2(1-u_\epsilon)^2.
\]
Using the assumption $\B\in \mathbf{\dot F^2}_\mu(\Omega)$ and the facts that $0\le1-u_\epsilon\le1$ and $|\nabla\eta|\le C$, we obtain
\begin{align*}
\int_{D(x,3)}|\nabla v_\epsilon|^2
\le C\int_{D(x,4)}|\nabla(1-u_\epsilon)|^2\eta^2+C.
\end{align*}
By~\eqref{eqn:minsol}, $1-u_\epsilon$ is a nonnegative subsolution in $B(x,4)$ with zero trace on $\Delta(x,4)$. Therefore, by the Caccioppoli inequality~\eqref{c_ineq},
\[
\int_{D(x,4)}|\nabla(1-u_\epsilon)|^2\eta^2
\le C\int_{D(x,4)}(1-u_\epsilon)^2|\nabla\eta|^2
\le C,
\]
which completes the proof of~\eqref{eqn:grad:bound}.
\end{proof}

\begin{rmk}\label{rmk:min:supersol}
Notice that the proof of~\eqref{eqn:minsol} shows the more general fact that, if $u$ is a nonnegative supersolution, then so is $\min(u,1)$.
\end{rmk}

Now, we have all the tools to obtain an estimate controlling the oscillation of solutions.

\begin{lem}\label{lem:osc}
Let $\Omega\subset\R^n$ be a one-sided chord-arc domain with unbounded complement.
Let $A:\Omega\to \R^{n\times n}$ be uniformly elliptic, and let $\B$ be a vector field satisfying $\B\in \mathbf{\dot F^2}_\mu(\Omega)$ and $(\Div \B)_+\in \mathbf{\dot F}_{\delta}(\Omega)$, with $\mu>0$ and $0<\delta<\lambda$, where $\lambda$ is the ellipticity constant of $A$.
Let $x\in\partial\Omega$, $0<5r\le\diam(\Omega)$, and let $u\in W^{1,2}(D(x,5r))$ be a bounded solution in $B(x,4r)$, that is,
\begin{equation}\label{eqn:sol}
\int_{D(x,4r)}A\nabla u\cdot\nabla\varphi+\int_{D(x,4r)}\B\cdot\nabla u\,\varphi=0
\end{equation}
for every $\varphi\in Y^{1,2}(\Omega)$ with $\Tr \varphi=0$ on $\Delta(x,4r)$ and $\varphi\equiv 0$ on $\Omega\setminus B(x,\rho)$ for some $0<\rho<4r$. Assume that $\Tr u$ is uniformly bounded on $\Delta(x,4r)$.

Then, there exists $\eta\in(0,1)$, depending only on $\mu$, $\delta$, $\lambda$, $n$, and the geometric constants of $\Omega$, such that
\begin{equation}\label{eqn:osc}
\osc_{D(x,r)} u\le \eta \osc_{D(x,4r)} u+(1-\eta)\osc_{\Delta(x,4r)}u.
\end{equation}
\end{lem}

\begin{proof}
The proof is the same as that of \cite[Lemma 8.15]{DavFM21}; therefore, we omit it. We only point out that, although \cite{DavFM21} deals with equations without lower-order terms in domains with higher-codimensional boundaries, the tools needed to prove \cite[Lemma 8.15]{DavFM21} are a density lemma for nonnegative supersolutions and the fact that, if $u$ is a supersolution, then so is $\min(u,1)$. In the present setting, these tools are provided by Lemma~\ref{lem:density} and Remark~\ref{rmk:min:supersol}.
\end{proof}

The inequality~\eqref{eqn:osc} can be iterated using \cite[Lemma 8.23]{GilT01}, which applies to functions satisfying a functional inequality such as~\eqref{eqn:osc}. Together with Lemma~\ref{lem:boundedness}, this yields H\"older continuity of solutions as a straightforward corollary of Lemma~\ref{lem:osc}.

\begin{thm}\label{thm:h}
Let $\Omega\subset\R^n$ be a one-sided chord-arc domain with unbounded complement.
Let $A:\Omega\to \R^{n\times n}$ be uniformly elliptic, and let $\B$ be a vector field satisfying $\B\in \mathbf{\dot F^2}_\mu(\Omega)$ and $(\Div \B)_+\in \mathbf{\dot F}_{\delta}(\Omega)$, with $\mu>0$ and $0<\delta<\lambda$, where $\lambda$ is the ellipticity constant of $A$.
Let $x\in\partial\Omega$, $0<5r\le\diam(\Omega)$, and let $u\in W^{1,2}(D(x,5r))$ be a solution in $B(x,4r)$, in the sense of~\eqref{eqn:sol}. If $\Tr u$ is bounded on $\Delta(x,4r)$ and $u$ is bounded in $D(x,2r)$, then there exists $\alpha>0$ such that, for every $0<s\le r$,
\begin{equation}\label{eqn:höld}
\osc_{D(x,s)}u\le C\biggl(\frac{s}{r}\biggr)^\alpha\osc_{D(x,r)}u+C\osc_{\Delta(x,\sqrt{sr})}\Tr u,
\end{equation}
where $C$ and $\alpha$ depend only on $\mu$, $\delta$, $\lambda$, $n$, and the geometric constants of $\Omega$.

If $\Tr u\equiv 0$ on $\Delta(x,4r)$, then $u$ is bounded in $D(x,2r)$ and, for every $0<s\le r$,
\begin{equation}\label{eqn:höld:1}
\osc_{D(x,s)}u\le C\biggl(\frac{s}{r}\biggr)^\alpha\biggl(\fint_{D(x,2r)}u^2\biggr)^{1/2},
\end{equation}
where $C$ depends only on $\mu$, $\delta$, $\lambda$, $n$, and the geometric constants of $\Omega$.
\end{thm}

\begin{cor}\label{cor:carl}
Let $\Omega\subset\R^n$ be a one-sided chord-arc domain with unbounded complement.
Let $A:\Omega\to \R^{n\times n}$ be uniformly elliptic, and let $\B$ be a vector field satisfying $\B\in \mathbf{\dot F^2}_\mu(\Omega)$ and $(\Div \B)_+\in \mathbf{\dot F}_{\delta}(\Omega)$, with $\mu>0$ and $0<\delta<\lambda$, where $\lambda$ is the ellipticity constant of $A$. Let $x\in\partial\Omega$, $0<7r\le\diam(\Omega)$, and let $u\in W^{1,2}(D(x,7r))$ be a nonnegative solution in $B(x,6r)$ with $\Tr u\equiv0$ on $\Delta(x,6r)$, in the sense of~\eqref{eqn:sol}. Then
\begin{equation}\label{eqn:hold:c}
u(X)\le C\Biggl(\frac{\delta(X)}{r}\Biggr)^\alpha\sup_{D(x,3r)}u,\qquad \forall\,X\in D(x,r),
\end{equation}
where $C$ depends only on $\mu$, $\delta$, $\lambda$, $n$, and the geometric constants of $\Omega$, and $\alpha$ is as in~\eqref{eqn:höld}.

Moreover, there exists a constant $K>1$, depending only on the geometric constants of $\Omega$ and on $n$, such that the following holds. If $u$ satisfies the assumptions above and is also a solution in the larger ball $B(x,Kr)$ with $\Tr u=0$ on $\Delta(x,Kr)$, then
\begin{equation}\label{eqn:carl}
u(X)\le C u(X_r),
\end{equation}
where $C$ depends only on $\mu$, $\delta$, $\lambda$, $n$, and the geometric constants of $\Omega$, and $X_r$ is a corkscrew point for $x$ at scale $r$.
\end{cor}

\begin{proof}
Let $X\in D(x,r)$. By the nonnegativity of $u$ and the fact that $\Tr u=0$ on $\Delta(x,Kr)$, if $K\ge 16$ we have $u(X)\le \osc_{D(z,2\delta(X))}u$, where $z\in \Delta(x,r+\delta(X))\subset\Delta(x,2r)$. Then, by~\eqref{eqn:höld:1},
\[
u(X)\le C\Biggl(\frac{\delta(X)}{r}\Biggr)^\alpha\osc_{D(z,2r)}u\le C\Biggl(\frac{\delta(X)}{r}\Biggr)^\alpha\sup_{D(x,8r)}u,
\]
which proves~\eqref{eqn:hold:c}. As noted in the proof of \cite[Lemma 3.11]{Hof26}, the proof of the analog of~\eqref{eqn:carl} in \cite[Theorem 1.1]{CafFMS81} can be modified to fit our setting, since the required ingredients are the existence of corkscrew points and Harnack chains, together with the application of~\eqref{eqn:hold:c} in small balls centered on the boundary. Therefore,~\eqref{eqn:carl} follows. The need for the scaling factor $K>1$ comes from the fact that Harnack chains connecting two points in $D(x,5r)$ whose distances from the boundary are comparable to $r$ may leave $D(x,5r)$, but may be chosen to remain in $D(x,Kr)$.
\end{proof}

\section{Harmonic measure}\label{sec:6}

We first show the validity of the maximum principle.

\begin{lem}\label{lem:max:princ}
Let $\Omega\subset\R^n$ be a one-sided chord-arc domain with unbounded complement.
Let $A:\Omega\to \R^{n\times n}$ be uniformly elliptic and let $\B$ be a vector field satisfying $\B\in \mathbf{\dot F^2}_\mu(\Omega)$ and $(\Div \B)_+\in \mathbf{\dot F}_{\delta}(\Omega)$, with $\mu>0$ and $0<\delta<\lambda$, where $\lambda$ is the ellipticity constant of $A$. Let $u\in Y^{1,2}(\Omega)$ be a supersolution, in the sense of Definition~\ref{def:sol}, and assume that $\Tr u\ge 0$ on $\partial\Omega$. Then $u\ge0$ in $\Omega$.
\end{lem}

\begin{proof}
It is equivalent to proving that $u_-\equiv 0$ in $\Omega$. We take $u_-$ as a test function in the equation satisfied by $u$.
Using also the ellipticity of $A$ and the fact that $(\Div\B)_+\in \mathbf{\dot F^1}_\delta(\Omega)$, we get
\begin{align*}
-\lambda\int_\Omega|\nabla u_-|^2&\ge\int_{\Omega}A\nabla u\cdot\nabla u_-\ge -\int_\Omega \B\cdot\nabla u\, u_-=\int_\Omega \B\cdot\nabla u_-\, u_-\\
&\ge-\int_{\Omega}(\Div\B)_+\frac{u_-^2}{2}\ge-\frac{\delta}{2}\int_\Omega|\nabla u_-|^2.
\end{align*}
Since $0<\delta<\lambda$, we infer that $\nabla u_-\equiv 0$ in $\Omega$, which, together with the fact that $\Tr u_-\equiv 0$, implies $u_-\equiv0$.
\end{proof}

Since solutions obey the maximum principle, we can define a harmonic measure for the operator with drift $L=-\Div(A\nabla)+\B\cdot \nabla$. Using H\"older continuity estimates (see Lemma~\ref{lem:h:int} and Theorem~\ref{thm:h}), we also prove the solvability of the continuous Dirichlet problem.

\begin{lem}\label{lem:hm}
Let $\Omega\subset\R^n$ be a one-sided chord-arc domain with unbounded complement.
Let $A:\Omega\to \R^{n\times n}$ be uniformly elliptic and let $\B$ be a vector field satisfying $\B\in \mathbf{\dot F^2}_\mu(\Omega)$ and $(\Div \B)_+\in \mathbf{\dot F}_{\delta}(\Omega)$, with $\mu>0$ and $0<\delta<\lambda$, where $\lambda$ is the ellipticity constant of $A$. Set $L=-\Div(A\nabla)+\B\cdot\nabla$.

There exists a family of Borel probability measures $\{\omega^Y_{L}\}_{Y\in\Omega}$, supported on $\partial\Omega$, such that the following holds.
For any $f\in \textup{Lip}_c(\partial\Omega)$, the unique solution $u_f\in Y^{1,2}(\Omega)$ to the Dirichlet problem
\begin{equation}\label{eqn:dir}
\begin{cases}
L u = 0 & \text{in } \Omega, \\
u = f & \text{on}\, \partial\Omega,
\end{cases}
\end{equation}
given by Lemma~\ref{lem:existence}, is continuous up to the boundary and represented by
\begin{equation}\label{eqn:rep:h}
u_f(Y)=\int_{\partial\Omega}f\,d\omega^Y_L.
\end{equation}
Moreover, if $f\in C_c(\partial\Omega)$, there exists a solution $u_f\in Y^{1,2}_{\textup{loc}}(\Omega)\cap C(\overline{\Omega})$ to~\eqref{eqn:dir} that satisfies~\eqref{eqn:rep:h}.
We say that $\omega_L$ is the harmonic measure for $L$. Whenever it is clear from the context that we are referring to the operator $L$, we omit the subscript, writing $\omega_L^X=\omega^X$.
\end{lem}

\begin{proof}
Let $f\in \textup{Lip}_c(\partial\Omega)$ and let $u_f$ be the solution to~\eqref{eqn:dir}. Applying  Lemma~\ref{lem:max:princ} to $(u_f-\|f\|_\infty)_+$ and $(-u_f-\|f\|_\infty)_+$, we find that $\|u_f\|_{L^\infty(\Omega)}\le\|f\|_\infty$. Therefore, for any $X\in\Omega$, the map $\textup{Lip}_c(\partial\Omega)\ni f\to u_f(X)$ is linear, positive, and continuous. By density of $\textup{Lip}_c(\partial\Omega)$ in $C_0(\partial\Omega)$, we extend the functional $\textup{Lip}_c(\partial\Omega)\ni f\to u_f(X)$ to a functional $T^X$ defined on $C_0(\partial\Omega)$. Then, by the Riesz representation theorem, there exists a family of nonnegative Borel measures $\{\omega^X\}_{X\in\Omega}$ such that
\begin{equation*}
T^X(f)=\int_{\partial\Omega}f\,d\omega^X,\quad X\in\Omega,\ f\in C_0(\partial\Omega).
\end{equation*}
In particular,~\eqref{eqn:rep:h} holds whenever $f\in\textup{Lip}_c(\partial\Omega)$, since in this case $T^X(f)=u_f(X)$. Since $\|T^X\|\le 1$, also $\omega^X(\partial\Omega)=\|T^X\|\le 1$.

Let now $f\in C_c(\partial\Omega)$ and let $\{f_i\}_{i\in\N}$ be a sequence of Lipschitz compactly supported functions converging uniformly to $f$ in $C(\partial\Omega)$. By Lemma~\ref{lem:h:int} and Theorem~\ref{thm:h}, $u_{f_i}\in C(\overline\Omega)$, and by the maximum principle
\[
\|u_{f_i}-u_{f_j}\|_{C(\overline\Omega)}
=\|u_{f_i-f_j}\|_{C(\overline{\Omega})}
\le\|f_i-f_j\|_\infty.
\]
This proves that $\{u_{f_i}\}_{i\in\N}$ is a Cauchy sequence in $C_b(\overline\Omega)$. Thus, there is a function $u\in C_b(\overline\Omega)$ such that $u_{f_i}$ converges uniformly to $u$ in $\overline\Omega$. In particular, $\|u\|_{L^\infty(\Omega)}\le\|f\|_{\infty}$ and $u|_{\partial\Omega}=f$. By the Caccioppoli inequality (see~\eqref{c_ineq:int}), for any ball $B$ such that $2B\subset\Omega$, we have
\[
\int_B|\nabla u_{f_i}-\nabla u_{f_j}|^2
=\int_B|\nabla u_{f_i-f_j}|^2
\le C_B\|u_{f_i-f_j}\|_{L^\infty(\Omega)}^2.
\]
Thus, $u_{f_i}$ converges in $W^{1,2}_{\textup{loc}}(\Omega)$ to some function $v$, and, by uniqueness of the limit, $u=v$. Therefore, taking the limit in the equation satisfied by $u_{f_i}$ (see~\eqref{eqn:sol:F}), we find that $u$ is a solution to $Lu=0$ (see Remark~\ref{rmk:sol}). Finally, $
\lim_{i\to\infty}\int_{\partial\Omega}f_i\,d\omega^X
=\int_{\partial\Omega}f\,d\omega^X,
$
extending the representation formula~\eqref{eqn:rep:h} to $u$.

We are left with proving that $\omega^X(\partial\Omega)\ge 1$ for all $X\in\Omega$. If $\Omega$ is bounded, then $f=1\in \textup{Lip}_c(\partial\Omega)$ and $u_f=1$. By~\eqref{eqn:rep:h}, $\omega^X(\partial\Omega)=1$. Suppose now that $\Omega$ is unbounded. Fix $X\in\Omega$, take $z\in\partial\Omega$, and let $\{g_i\}_{i\in\N}\subset \textup{Lip}_c(\partial\Omega)$ be a sequence such that $0\le g_i\le 1$, $g_i=1$ in $\Delta(z,i)$, and $g_i=0$ in $\partial\Omega\setminus \Delta(z,2i)$. Then, for $i$ large enough,~\eqref{eqn:höld}, together with $0\le u_{g_i}\le 1$, gives
\[
|u_{g_i}(X)-1|
=|u_{g_i}(X)-u_{g_i}(z)|
\le\osc_{D(z,|z-X|)}u_{g_i}
\le C\Biggl(\frac{|z-X|}{i}\Biggr)^\alpha\osc_{D(z,i)}u_{g_i}
\le C\Biggl(\frac{|z-X|}{i}\Biggr)^\alpha.
\]
Taking the limit as $i\to\infty$, we find $\lim_{i\to\infty}u_{g_i}(X)=1$. Since $g_i\to1$ pointwise on $\partial\Omega$,~\eqref{eqn:rep:h} and the dominated convergence theorem give $\omega^X(\partial\Omega)=1$.
\end{proof}

We now show that Bourgain's estimate also holds in our setting.

\begin{lem}\label{lem:bourgain}
Let $\Omega\subset\R^n$ be a one-sided chord-arc domain with unbounded complement.
Let $A:\Omega\to \R^{n\times n}$ be uniformly elliptic and let $\B$ be a vector field satisfying $\B\in \mathbf{\dot F^2}_\mu(\Omega)$ and $(\Div \B)_+\in \mathbf{\dot F}_{\delta}(\Omega)$, with $\mu>0$ and $0<\delta<\lambda$, where $\lambda$ is the ellipticity constant of $A$. Set $L=-\Div(A\nabla)+\B\cdot\nabla$. There exists $c\in (0,1)$, depending only on $\mu$, $\delta$, $\lambda$, $n$, and the geometric constants of $\Omega$, such that, for any $x\in\partial\Omega$, $0<2r\le\diam(\Omega)$, and $X\in D(x,r/4)$, we have
\begin{equation}\label{eqn:bourgain}
\omega^X(\Delta(x,r))\ge c.
\end{equation}
\end{lem}

\begin{proof}
 Let $i\in\N$ and let $f_i\in C^\infty_c(\R^n)$ be such that $0\le f_i\le1$, $f_i=1$ in $B(x,r(1-2^{-i}))$, and $f_i=0$ outside $B(x,r)$. Let $u_{f_i}$ be the solution to~\eqref{eqn:dir} with $f=f_i$. By~\eqref{eqn:density},
\[
u_{f_i}(X)\ge c,\qquad \forall X\in D(x,r(1-2^{-i})/4),
\]
where $c$ depends only on $\mu$, $\delta$, $\lambda$, $n$, and the geometric constants of $\Omega$. Taking the limit as $i\to\infty$, and using that
\[
u_{f_i}(X)=\int_{\Delta(x,r)}f_i\,d\omega^X,
\]
we obtain~\eqref{eqn:bourgain}.
\end{proof}

 \section{Green's function}\label{sec:7}
We now show existence and properties of the Green's function for the operator $L=-\Div(A\nabla)+\B\cdot\nabla$. In doing so, we need the more restrictive assumption
\[
\Div\B\le 0.
\]
More precisely, for the Green's function we construct to satisfy the corresponding distributional identity with a Dirac mass at the pole, we will assume either that $\B\in L^{n,1}_{\textup{loc}}(\Omega)$ or that $\Div\B=0$. Without one of these two assumptions, we were unable to make sense of the integral in~\eqref{eqn:green}. In the first case, the estimates we show won't depend on any Lorentz norm of $\B$.

Finally, to have the appropriate lower bound for the Green's function, we will need to assume
\[
\Div\B=0.
\]
The need for more restrictive assumptions on $\Div\B$ is also present in \cite{LisZ04}, where the authors proved two-sided Gaussian bounds for the heat kernel of the parabolic operator $\partial_t-\Delta+\B\cdot\nabla$. For the upper bound, the authors assumed a global form-boundedness condition on $\Div\B$ and an additional smallness condition on $(\Div\B)_+$. Whereas for the lower bound, they assumed $\Div\B=0$. See also \cite{KinS22} for two-sided Gaussian bounds of the heat kernel with (multiplicative) form-bounded drift $\B$ and form-bounded divergence $\Div\B$ satisfying some other smallness condition. In the elliptic setting, Green's functions for operators with singular drifts were studied in \cite{KimS19}, \cite{Mou23}, \cite{Sak21}, \cite{Pat24}, and \cite{Hof26}. The need for the condition $\B\in L^{n,1}_{\textup{loc}}(\Omega)$ is exploited in~\cite{Sak21}, where the author shows an example of a vector field $\B$ that violates such a condition for which there cannot exist a Green's function for the operator $-\Delta+\B\cdot\nabla$ that satisfies both~\eqref{eqn:green} and~\eqref{eqn:green:sup}. In \cite{Hof26}, the author proves two-sided local bounds for the Green's function under the smallness condition~\eqref{eqn:small}. In \cite{Pat24}, the author shows that, if the constant $\epsilon$ in~\eqref{eqn:small} is not small, the local upper bound for the Green's function can fail. Moreover, the drift $\B$ constructed in \cite{Pat24} satisfies $\Div\B>0$. By the condition $\Div\B\le 0$, we can guarantee that the upper bound holds.

For bounded domains and divergence-free drifts, the weak Lorentz
estimates for the approximate Green functions below
(see~\eqref{eqn:weak:g} and~\eqref{eqn:weak:grad:g}) follow by
applying \cite[Lemma~15]{Har16} to their regularized sources.
Their interior pointwise upper bound then follows from the weak
estimate for the functions and interior elliptic estimates.
The local pointwise lower bound~\eqref{eqn:inf:g} can likewise
be deduced by applying \cite[Theorem~1]{Har16} to the approximate
Green functions.
We provide a direct proof of this lower bound.
Hara does not carry out a Green-function construction under
these drift assumptions; see \cite[Remark~2]{Har16}.
\begin{thm}\label{thm:green}
Let $\Omega\subset\R^n$ be a one-sided chord-arc domain with unbounded complement.
Let $A:\Omega\to \R^{n\times n}$ be uniformly elliptic with ellipticity constant $\lambda$ and let $\B$ be a vector field satisfying $\B\in \mathbf{\dot F^2}_\mu(\Omega)$, for some $0<\mu<\infty$, and $\Div\B\le 0$ in the weak sense.
Denote by $L$ the operator given by $L=-\Div(A\nabla)+\B\cdot\nabla$. 

Then, for any $Y\in\Omega$, there is a nonnegative function $G^Y_L=G^Y$ such that, for all $\epsilon>0$, $G^Y\in C(\overline{\Omega}\setminus\{Y\})\cap Y^{1,2}(\Omega\setminus B(Y,\epsilon))$, such that $G^Y(x)=0$ for all $x\in\partial\Omega$, and that has the following properties. 

$G^Y$ is a solution to $LG^Y=0$ in $\Omega\setminus\{Y\}$, that is,

\begin{equation}\label{eqn:green:sol:away}
\int_{\Omega}A\nabla G^Y\cdot\nabla\varphi+\int_{\Omega}\B\cdot\nabla G^Y\varphi=0,\qquad\forall\varphi\in C^\infty_c(\Omega\setminus{Y}).
\end{equation}
$G^Y$ satisfies the upper bound
\begin{equation}\label{eqn:green:sup}
0\le G^Y(X)\le C|X-Y|^{2-n},\qquad\forall X\in\Omega\setminus\{Y\},
\end{equation}
and the gradient bounds
\begin{equation}\label{eqn:weak:green}
\nabla G^Y\in L^{\frac{n}{n-1},\infty}(\Omega),
\end{equation}
\begin{equation}\label{eqn:green:grad}
\int_{\Omega\setminus B(Y,R)}|\nabla G^Y|^2\le C R^{2-n},\qquad\forall R>0,
\end{equation}
where $C$ is a constant depending only on $n$, $\mu$, $\lambda$, and the geometric constants of $\Omega$.
\begin{equation}\label{eqn:green:grad:1}
\int_{B(Y,R)\cap\Omega}|\nabla G^Y|^q\le C R^{n-q(n-1)},\qquad\forall 1\le q<\frac{n}{n-1},
\end{equation}
where $C$ is a constant depending only on $n$, $\mu$, $\lambda$, $q$, and the geometric constants of $\Omega$.

Let $L^*$ be the adjoint of the operator $L=-\Div(A\nabla)+\B\cdot\nabla$, let $F\in L^\infty_c(\overline{\Omega})$, and let $u^*\in Y^{1,2}_0(\Omega)$ be the solution to $L^*u=F$ (see Remark~\ref{rmk:adjoint}). Then,
\begin{equation}\label{eqn:rep}
u^*(Y)=\int_{\Omega}G^Y(X)F(X)\,dX,
\end{equation}
for almost every $Y\in\Omega$.

Moreover, if $\B\in L^{n,1}_{loc}(\Omega)$, then 
\begin{equation}\label{eqn:green}
\int_{\Omega}A\nabla G^Y\cdot\nabla\varphi+\int_{\Omega}\B\cdot\nabla G^Y\varphi=\varphi(Y),\qquad\forall\varphi\in C^\infty_c(\Omega).
\end{equation}
If instead $\Div\B=0$, then 
\begin{equation}\label{eqn:green:divfree}
\int_{\Omega}A\nabla G^Y\cdot\nabla\varphi-\int_{\Omega}\B\cdot\nabla\varphi G^Y=\varphi(Y),\qquad\forall\varphi\in C^\infty_c(\Omega).
\end{equation}

If $\Div\B=0$, $G^Y$ has additional properties. First, the following local lower bound holds.
\begin{equation}\label{eqn:green:inf}
G^Y(X)\ge C^{-1}|X-Y|^{2-n}, \qquad\forall X\ne Y:|X-Y|<\delta(Y)/2,
\end{equation}
where $C$ is a constant depending only on $n$, $\mu$, $\lambda$, and the geometric constants of $\Omega$.

Moreover, for any $\Phi\in \textup{Lip}_c(\R^n)$ let $u_{\Phi}$ be the solution to $Lu_\Phi=0$ in $\Omega$ satisfying $\Tr u_\Phi=\Phi$ on $\partial\Omega$. Then, for all $X\in\Omega$, we have
\begin{equation}\label{eqn:rep:1}
u_\Phi(X)-\Phi(X)=-\int_{\Omega}A^T\nabla G^X_{L^*}\nabla \Phi-\int_\Omega\B\cdot\nabla \Phi G^X_{L^*}.
\end{equation}
Finally, for any $X\ne Y$,
\begin{equation}\label{eqn:sym}
G^Y_L(X)=G^X_{L^*}(Y).
\end{equation}
\end{thm}

\begin{proof}
We adapt the arguments in \cite{GruW82} to our setting.
Fix $Y\in\Omega$. For any $0<\rho<\delta(Y)/20$, let $g_\rho^Y$ be the unique solution of~\eqref{eqn:sol:F}, with $\langle F,\varphi\rangle=\fint_{B(Y,\rho)}\varphi$ for any $\varphi\in Y^{1,2}_0(\Omega)$, i.e.,
\begin{equation}\label{eqn:sol:grho}
\int_{\Omega}A\nabla g^Y_\rho\cdot\nabla\varphi+\int_{\Omega}\B\cdot\nabla g^Y_\rho\varphi=\fint_{B(Y,\rho)}\varphi,
\quad\forall\varphi\in Y^{1,2}_0(\Omega).
\end{equation}
The existence and uniqueness of $g^Y_\rho$ are provided by Lemma~\ref{lem:existence}, since
\[
\left|\fint_{B(Y,\rho)}\varphi\right|\le C\|\varphi\|_{Y^{1,2}(\Omega)}\rho^{1-\frac{n}{2}}.
\]
Since $g^Y_\rho$ is a supersolution with zero trace, then $g^Y_\rho\ge 0$ by Lemma~\ref{lem:max:princ}.
Set $\Omega_t:=\{g^Y_\rho>t\}$. On $\Omega_{2t}$ we have $\log\Bigl(\frac{g^Y_\rho}{t}\Bigr)\ge\log(2)$. Thus, setting $w_t:=\Bigl[\log\Bigl(\frac{g^Y_\rho}{t}\Bigr)\Bigr]_+\in Y^{1,2}_0(\Omega)$,
\begin{equation}\label{eqn:meas:g}
|\Omega_{2t}|\le C\int_{\Omega_{2t}}w_t^{\frac{2n}{n-2}}\le C\biggl(\int_{\Omega}|\nabla w_t|^2\biggr)^{\frac{n}{n-2}}
\end{equation}
Using the ellipticity of $A$ and the fact that $w_t=0$ in $\Omega\setminus\Omega_t$ we get
\begin{equation*}
\int_{\Omega}|\nabla w_t|^2\le C\int_{\Omega_t}A\nabla g^Y_\rho\nabla v_t,
\end{equation*}
where $v_t:=\Bigl(\frac{1}{t}-\frac{1}{g^Y_\rho}\Bigr)_+$.
Using $v_t$ as a test function in~\eqref{eqn:sol:grho} we obtain
\begin{equation}\label{eqn:meas:g:1}
\int_{\Omega}|\nabla w_t|^2\le C\fint_{B(Y,\rho)}v_t-C\int_{\Omega}\B\cdot\nabla g^Y_\rho v_t
\le Ct^{-1}-C\int_{\Omega}\B\cdot\nabla(g^Y_\rho v_t-w_t),
\end{equation}
where we used that $0\le v_t\le t^{-1}$ and that $\nabla(g^Y_\rho v_t-w_t)=\nabla g^Y_\rho v_t$. We now claim that $-\int_{\Omega}\B\cdot\nabla(g^Y_\rho v_t-w_t)\le0$.
Notice that, even though $0\le g^Y_\rho v_t-w_t\in Y^{1,2}_0(\Omega)$ and $(\Div\B)_+=0$, the low regularity of $\B$ does not allow us to conclude immediately. For instance, we know that $\B\cdot\nabla(g^Y_\rho v_t-w_t)\in L^1(\Omega)$ only because $\B\cdot\nabla(g^Y_\rho v_t-w_t)=\B\cdot\nabla g^Y_\rho v_t$ (see Remark~\ref{rmk:bil}). To prove the claim, we will rely on~\eqref{eqn:beta:0}. Set $f:\R\to\R$ by $f(s):=(s-1)_+-(\log s)_+$, so that $f\Bigl(\frac{g^Y_\rho}{t}\Bigr)=g^Y_\rho v_t-w_t$. Since, for $s\ge 1$, we have $(s-1)-\log s\ge \frac{(s-1)^2}{2s^2}$, then $(\sqrt f)'(s)\le \frac{\sqrt{2}}{2}$. Therefore, since $f(0)=0$ and $\sqrt{f}$ is Lipschitz, we have that $\sqrt{g^Y_\rho v_t-w_t}=\sqrt{f\Bigl(\frac{g^Y_\rho}{t}\Bigr)}\in Y^{1,2}_0(\Omega)$, which allows us to write
\[
-\int_{\Omega}\B\cdot\nabla(g^Y_\rho v_t-w_t)=-2\int_{\Omega}\B\cdot\nabla \sqrt{f\Bigl(\frac{g^Y_\rho}{t}\Bigr)}\sqrt{f\Bigl(\frac{g^Y_\rho}{t}\Bigr)}\le 0.
\]
Then, by~\eqref{eqn:meas:g:1},
\begin{equation*}
\int_{\Omega_{2t}}\frac{|\nabla g^Y_\rho|^2}{(g^Y_\rho)^2}\le C t^{-1},
\end{equation*}
which, in view of~\eqref{eqn:meas:g}, implies
\begin{equation}\label{eqn:weak:g}
\|g^Y_\rho\|_{L^{\frac{n}{n-2},\infty}(\Omega)}\le C,
\end{equation}
where $C$ is a constant depending only on $n$ and $\lambda$.
Now, we show how to derive~\eqref{eqn:green:sup} from~\eqref{eqn:weak:g}. Let $X\in \Omega$ be such that $|X-Y|\ge 4\rho$. Then, setting $r=|X-Y|/8$, if $\delta(X)\ge r$, we apply Lemma~\ref{lem:boundedness:int} at scale $r/4$. If $\delta(X)<r$ and $x\in\partial\Omega$ is such that $|X-x|<r$, we apply Lemma~\ref{lem:boundedness} at scale $r$. Thus, distinguishing between the two cases, we get
\[
g^Y_\rho(X)\le C \fint_{D_{r}} g^Y_\rho\le Cr^{2-n}\|g^Y_\rho\|_{L^{\frac{n}{n-2},\infty}(\Omega)},
\]
where $D_r=B(X,r/2)$ in the first case and $D_r=D(x,2r)$ in the second case.
By~\eqref{eqn:weak:g} and the definition of $r$ we have
\begin{equation}\label{eqn:g:up}
0\le g^Y_\rho(X)\le C|X-Y|^{2-n}\qquad \forall X:|X-Y|\ge 4\rho.
\end{equation}
We now move to the proof of~\eqref{eqn:green:grad} and~\eqref{eqn:green:grad:1}. Let $R\ge4\rho$ and let $\eta\in C^\infty(\R^n)$ be a nonnegative function such that $\eta=0$ in $B(Y,R)$, $\eta=1$ in $\R^n\setminus B(Y,2R)$ and $|\nabla\eta|\le 2/R$. Then, using $g^Y_\rho\eta^2$ as a test function in~\eqref{eqn:sol:grho}, we get
\[
\int_{\Omega}A\nabla g^Y_\rho\nabla g^Y_\rho\eta^2\le CR^{-2}\int_{\Omega\cap B(Y,2R)\setminus B(Y,R)}|g^Y_\rho|^2-\int_{\Omega}\B\cdot\nabla g^Y_\rho g^Y_\rho\eta^2.
\]
Using the ellipticity of $A$ and the hypothesis $\Div B\le 0$, we have
\[
\int_{\Omega}|\nabla g^Y_\rho|^2\eta^2\le CR^{-2}\int_{\Omega\cap B(Y,2R)\setminus B(Y,R)}|g^Y_\rho|^2+C\int_\Omega |\B|(g^Y_\rho)^2|\nabla \eta|\eta.
\]
As usual, the hypothesis $\B\in \mathbf{\dot F^2}_\mu(\Omega)$ and Young's inequality allow us to split the second term on the right-hand side and to hide $\epsilon \int |\nabla g^Y_\rho|^2\eta^2$ in the left-hand side. This, together with~\eqref{eqn:g:up}, leaves us with
\begin{equation}\label{eqn:grad:g}
\int_{\Omega\setminus B(Y,2R)}|\nabla g^Y_\rho|^2\le CR^{2-n}.
\end{equation}
Moreover, by~\eqref{eqn:est:ex}, we have
\begin{equation*}
\int_\Omega|\nabla g^Y_\rho|^2\le C\rho^{2-n},
\end{equation*}
which, in particular, proves~\eqref{eqn:grad:g} also for $R<4\rho$.
We now estimate the measure of the superlevel set $\{|\nabla g^Y_\rho|>t\}$ for $t>0$. Markov inequality and~\eqref{eqn:grad:g}, applied with $R=t^{-\frac{1}{n-1}}$, give
\[
\bigl|\{|\nabla g^Y_\rho|>t\}\bigr|\le CR^n+\int_{\Omega\setminus B(Y,2R)}|\nabla g^Y_\rho|^2t^{-2}\le CR^n+C t^{-2}R^{2-n}=Ct^{-\frac{n}{n-1}},
\]
that is,
\begin{equation}\label{eqn:weak:grad:g}
\|\nabla g^Y_\rho\|_{L^{\frac{n}{n-1},\infty}(\Omega)}\le C.
\end{equation}
Therefore, for all $1\le q<\frac{n}{n-1}$ and all $R>0$,
\begin{equation}\label{eqn:grad:g:q}
\int_{\Omega\cap B(Y,R)}|\nabla g^Y_\rho|^q\le CR^{n-q(n-1)},
\end{equation}
where $C$ depends also on $q$.

Since $g^Y_\rho$ is a solution in $\Omega\setminus B(Y,2\rho)$ with zero trace on $\partial\Omega$, by Lemma~\ref{lem:h:int}, Theorem~\ref{thm:h}, and~\eqref{eqn:g:up}, for any $\epsilon>0$ the family $\{g^Y_\rho\}_{0<\rho<\epsilon/2}$ is equi-H\"older continuous and equibounded on any compact set $K\subset \overline{\Omega}\setminus B(Y,2\epsilon)$. By the Arzel\'a-Ascoli theorem and using a diagonal argument, up to a subsequence and as $\rho\to0$, $g^Y_\rho$ converges locally uniformly in $\overline{\Omega}\setminus \{Y\}$ to a function $G^Y\in C(\overline{\Omega}\setminus\{Y\})$. In particular, since $g^Y_\rho|_{\partial\Omega}=0$, also $G^Y|_{\partial\Omega}=0$. From now on, we will abuse notation by saying that we take the limit as $\rho\to 0$ of quantities involving $g^Y_\rho$; instead, we will always mean that we take the limit along some subsequence $\rho_\eta\to 0$. By~\eqref{eqn:grad:g}, ~\eqref{eqn:grad:g:q}, and weak compactness, $G^Y$ is a Sobolev function and,  $\nabla g^Y_\rho$ converges weakly in $L^2_{\textup{loc}}(\Omega\setminus\{Y\})$ and weakly in $L^1(\Omega)$ to $\nabla G^Y$ as $\rho\to 0$. Then, properties~\eqref{eqn:green:sol:away},  \eqref{eqn:green:sup}, ~\eqref{eqn:weak:green}, \eqref{eqn:green:grad}, and~\eqref{eqn:green:grad:1} follow by taking the limit as $\rho\to 0$ in~\eqref{eqn:sol:grho}, \eqref{eqn:g:up}, \eqref{eqn:grad:g}, and~\eqref{eqn:grad:g:q}.

Let now $F\in L^\infty_c(\overline{\Omega})$ and let $u^*\in Y^{1,2}_0(\Omega)$ be the unique solution to $L^*u^*=F$. Let $Y$ be a Lebesgue point of $u^*$. Using $u^*$ as a test function in~\eqref{eqn:sol:grho}, and $g^Y_\rho$ as a test function in the equation satisfied by $u^*$, we have
\[
\fint_{B(Y,\rho)}u^*=\langle Lg^Y_\rho, u^*\rangle=\langle g^Y_\rho, L^*u^*\rangle=\int_{\Omega}g^Y_\rho F.
\]
By~\eqref{eqn:weak:g}, \eqref{eqn:grad:g:q}, and the compact Sobolev embedding, $g^Y_\rho\to G^Y$ strongly in $L^1$ close to $Y$, while the convergence is locally uniform away from $Y$; since $F\in L^\infty_c(\overline{\Omega})$, letting $\rho\to0$ yields~\eqref{eqn:rep}.

It is worth explaining in detail the justification of the passage to the limit to obtain either~\eqref{eqn:green} or~\eqref{eqn:green:divfree}, since this is the only point in which either  $\B\in L^{n,1}_{\textup{loc}}(\Omega)$ or $\Div\B=0$ is needed.
Assume that $\B\in L^{n,1}_{\textup{loc}}(\Omega)$. By the Hölder inequality in Lorentz spaces, we have, for any $0<r<\delta(Y)/2$, 
\[
\int_{B(Y,r)}|\B||\nabla G^Y_\rho|\le C \|B\|_{L^{n,1}(B(Y,r))}\|\nabla G^Y\|_{L^{\frac{n}{n-1},\infty}(\Omega)}.
\]
Therefore, $\B\cdot\nabla G^Y_\rho\in L^1_{\textup{loc}}(\Omega)$. Let now $\varphi\in C^\infty_c(\Omega)$, use it as a test function in~\eqref{eqn:sol:grho}. We already know that, as $\rho\to0$, $\int_{\Omega}A\nabla g^Y_\rho\nabla\varphi\to\int_{\Omega}A\nabla G^Y\nabla\varphi$ and $\fint_{B(Y,\rho)}\varphi\to\varphi(Y)$. Thus, in order to prove~\eqref{eqn:green} it sufficies to prove 
\begin{equation}\label{eqn:lor:lim}
\int_{\Omega}\B\cdot(\nabla G^Y-\nabla g^Y_\rho)\varphi\to0.
\end{equation}
For $0<r<\delta(Y)/4$, let $\eta_r$ be a smooth function supported in $B(Y,2r)$, with $\eta_r=1$ in $B(Y,r)$. Since $\nabla g^Y_\rho\to \nabla G^Y$ weakly in $L^2_{\textup{loc}}(\Omega\setminus\{Y\})$, using Hölder inequality in Lorentz spaces again, together with~\eqref{eqn:weak:grad:g},  we have

\begin{align*}
\limsup_{\rho\to 0}
\left|
\int_{\Omega}
\B\cdot(\nabla G^Y-\nabla g^Y_\rho)\,\varphi
\right|
&\le
\limsup_{\rho\to 0}
\left|
\int_{\Omega}
\B\cdot(\nabla G^Y-\nabla g^Y_\rho)\,\varphi\eta_r
\right|
\\
&\quad+
\limsup_{\rho\to 0}
\left|
\int_{\Omega}
\B\cdot(\nabla G^Y-\nabla g^Y_\rho)
(1-\eta_r)\varphi
\right|
\\
&=
\limsup_{\rho\to 0}
\left|
\int_{\Omega}
\B\cdot(\nabla G^Y-\nabla g^Y_\rho)\,\varphi\eta_r
\right|
\\
&\le
C\|\varphi\|_{L^\infty(\Omega)}
\limsup_{\rho\to 0}
\left(
\|\nabla G^Y\|_{L^{\frac n{n-1},\infty}(\Omega)}
+
\|\nabla g^Y_\rho\|_{L^{\frac n{n-1},\infty}(\Omega)}
\right)
\\
&\qquad\times
\|\B\|_{L^{n,1}(B(Y,2r))}
\\
&\le
C\|\varphi\|_{L^\infty(\Omega)}
\|\B\|_{L^{n,1}(B(Y,2r))}.
\end{align*}

Since $r$ is arbitrarily small, this proves~\eqref{eqn:lor:lim}.

Assume  that $\Div\B=0$ instead. Without stronger integrability assumptions on $\B$, we cannot conclude that $\B\cdot\nabla G^Y$ is integrable close to the pole $ Y$. Therefore, the integral in~\eqref{eqn:green} might not make sense. However, we can use the divergence-free condition to rewrite the equation satisfied by $g^Y_\rho$ as 
\[
\int_{\Omega}A\nabla g^Y_\rho\cdot\nabla\varphi-\int_{\Omega}\B\cdot\nabla \varphi g^Y_\rho=\fint_{B(Y,\rho)}\varphi,
\quad\forall\varphi\in Y^{1,2}_0(\Omega).
\]
For any $\varphi\in C^\infty_c(\Omega)$, we can take the limit as $\rho\to 0$, getting~\eqref{eqn:green:divfree}.

We now prove a local lower bound for $g^Y_\rho$.
Let $X\in\Omega$ with $4\rho<|X-Y|<\delta(Y)/2$, set $r=5|X-Y|/4$, and let $0\le\eta\le 1$ be a smooth function with $\eta=1$ in $B(Y,r)\setminus B(Y,r/2)$, $\eta=0$ outside of $B(Y,3r/2)\setminus B(Y,r/4)$, and $|\nabla \eta|\le 10/r$. Using $g^Y_\rho\eta^2$ as a test function in~\eqref{eqn:sol:grho} we have
\begin{align*}
\int_{\Omega}|\nabla g^Y_\rho|^2\eta^2&\le C \int_\Omega A\nabla g^Y_\rho\nabla (g^Y_\rho \eta^2)- C\int_\Omega A\nabla g^Y_\rho \nabla \eta \eta g^Y_\rho\\
&=-C\int_{\Omega}\B\cdot\nabla g^Y_\rho g^Y_\rho\eta^2+C\int_\Omega |\nabla g^Y_\rho|| \nabla \eta| \eta g^Y_\rho.
\end{align*}
Then, using, as usual, Young's inequality and the hypothesis on $\B$, we write
\[
\int_{B(Y,r)\setminus B(Y,r/2)}|\nabla g^Y_\rho|^2\le \int_{\Omega}|\nabla g^Y_\rho|^2\eta^2\le C\int_{\Omega} |g^Y_\rho|^2|\nabla\eta|^2\le Cr^{-2}\int_{B(Y,3r/2)\setminus B(Y,r/4)}|g^Y_\rho|^2.
\]
Since $|Z-Y|\approx r$ in $B(Y,3r/2)\setminus B(Y,r/4)$, we have
\begin{equation}\label{eqn:pre:inf:g}
\int_{B(Y,r)\setminus B(Y,r/2)}|\nabla g^Y_\rho|^2\le Cr^{n-2}\sup_{B(Y,3r/2)\setminus B(Y,r/4)}(g^Y_\rho)^2\le Cr^{n-2}(g^Y_\rho(X))^2,
\end{equation}
where we applied Lemma~\ref{lem:h} in small balls, with radius comparable to $r$, covering the annulus $B(Y,3r/2)\setminus B(Y,r/4)$.
Assume now that $\Div\B=0$. Let $\varphi\in C^\infty_c(\Omega)$ be such that $0\le\varphi\le 1$, $\varphi=1$ on $B(Y,r/2)$, $\varphi=0$ in $\Omega\setminus B(Y,r)$, $|\nabla \varphi|\le \frac{10}{r}$, and $|\nabla^2\varphi|\le Cr^{-2}$. Using $\varphi$ as a test function in~\eqref{eqn:sol:grho} we have
\begin{align*}
1&=\int_{B(Y,r)\setminus B(Y,r/2)}A\nabla g^Y_\rho\nabla \varphi+\int_{\Omega}\B\cdot\nabla g^Y_\rho \varphi\\
&=\int_{B(Y,r)\setminus B(Y,r/2)}A\nabla g^Y_\rho\nabla \varphi-\int_{B(Y,r)\setminus B(Y,r/2)}\B\cdot\nabla\varphi g^Y_\rho\\
&\le Cr^{-1+n/2}\Bigl(\int_{B(Y,r)\setminus B(Y,r/2)}|\nabla g^Y_\rho|^2\Bigr)^{1/2}\\
&\quad+\Bigl(\int_{B(Y,r)\setminus B(Y,r/2)}|\B|^2|\nabla\varphi|^2\Bigr)^{1/2}
\Bigl(\int_{B(Y,r)\setminus B(Y,r/2)}|g^Y_\rho|^2\Bigr)^{1/2}\\
&\le Cr^{-1+n/2}\Bigl(\int_{B(Y,r)\setminus B(Y,r/2)}|\nabla g^Y_\rho|^2\Bigr)^{1/2}
+Cr^{n-2}\sup_{B(Y,r)\setminus B(Y,r/2)}g^Y_\rho\\
&\le Cr^{n-2}g^Y_\rho(X),
\end{align*}
where in the second line we used the hypothesis $\Div \B=0$, while in the fourth line we applied the hypothesis $\B\in \mathbf{\dot F^2}_\mu(\Omega)$, using as test functions in ~\eqref{eqn:fb} the partial derivatives of $\varphi$. In the last line we used~\eqref{eqn:pre:inf:g} and Lemma~\ref{lem:h}. We conclude
\begin{equation}\label{eqn:inf:g}
g^Y_\rho(X)\ge C^{-1}|X-Y|^{2-n},\qquad\forall X:4\rho<|X-Y|<\delta(Y)/2.
\end{equation}
Taking the limit as $\rho\to 0$ in~\eqref{eqn:inf:g} yields~\eqref{eqn:green:inf}.
Let $\Phi\in \textup{Lip}_c(\R^n)$ and let $u_\Phi\in Y^{1,2}(\Omega)$ be the solution to $Lu_\Phi=0$ with $\Tr u_\Phi=\Phi$.
Since $\Div \B=0$, then we can write
\[
L^*=-\Div(A^T\nabla)-\B\cdot\nabla,
\]
and hence $L^*$ satisfies the assumptions of Theorem~\ref{thm:green}. For $X\in\Omega$ let $g^X_{\rho,L^*}\in Y^{1,2}_0(\Omega)$ be the solution of~\eqref{eqn:sol:grho} with $X$ in place of $Y$, $A^T$ in place of $A$, and $-\B$ in place of $\B$.
Using $u_\Phi-\Phi\in Y^{1,2}_0(\Omega)$ as a test function in the equation satisfied by $g^X_{\rho,L^*}$ and $g^X_{\rho,L^*}$ as a test function in the equation satisfied by $u_\Phi$ we get
\begin{align*}
\fint_{B(X,\rho)}u_\Phi-\Phi&=\langle L^*g^X_{\rho,L^*},u_\Phi-\Phi\rangle=\langle L(u_\Phi-\Phi), g^X_{\rho,L^*} \rangle\\
&=-\langle L\Phi, g^X_{\rho,L^*} \rangle=-\int_{\Omega}A^T\nabla g^X_{\rho,L^*}\nabla \Phi-\int_\Omega\B\cdot\nabla \Phi g^X_{\rho,L^*}.
\end{align*}
Taking the limit as $\rho\to0$ we obtain~\eqref{eqn:rep:1}.
Take now any $X\ne Y$. Using $g^X_{\tilde\rho,L^*}$ as a test function in the equation satisfied by $g^Y_{\rho,L}$ and vice versa, we get
\begin{align*}
\fint_{B(Y,\rho)}g^X_{\tilde\rho,L^*}=\langle L g^Y_{\rho, L}, g^X_{\tilde\rho, L^*} \rangle=\langle L^*g^X_{\tilde\rho, L^*}, g^Y_{\rho, L}\rangle= \fint_{B(X,\tilde\rho)}g^Y_\rho.
\end{align*}
Taking the limit as $\rho,\tilde\rho\to0$, we end the proof of~\eqref{eqn:sym}.
\end{proof}

\begin{rmk}
The additional assumption $\B\in L^{n,1}_{\textup{loc}}(\Omega)$ that we imposed to get~\eqref{eqn:green} is not artificial. Reversing the sign of the drift in Sakellaris's example
\cite[Proposition~6.2]{Sak21}, consider, for $n\ge3$ and $\epsilon>0$,
\[
 \B(X):=
\epsilon\frac{X}{|X|^2\log(|X|)},
\qquad X\in\Omega:=B(0,e^{-1}),\quad X\ne 0.
\]
Then $\Div\B\le0$,
and $\B\in L^{n,q}(\Omega)$ for every $q>1$, whereas
$\B\notin L^{n,1}_{\textup{loc}}(\Omega)$.
A direct computation shows that the approximate Green functions with pole at $0$,
$g^0_\rho$, defined by~\eqref{eqn:sol:grho}, converge to
$G^0\equiv0$ in $\Omega\setminus\{0\}$. Consequently, this limit
does not satisfy the standard weak formulation
\[
\int_\Omega \nabla G^0\cdot\nabla\varphi
+\int_\Omega(\mathbf B\cdot\nabla G^0)\varphi
=\varphi(0),
\qquad \varphi\in C_c^\infty(\Omega).
\]
\end{rmk}

\section{CFMS estimates and comparison principle}\label{sec:8}
Throughout this section, for any $Y\in\Omega$ we denote 
\begin{equation*}
\Delta_Y:=B(Y,10\delta(Y))\cap\partial\Omega.
\end{equation*}
We begin by showing the connection between the Green's function and the harmonic measure.

\begin{lem}
Let $\Omega\subset\R^n$ be a one-sided chord-arc domain with unbounded complement.
Let $A:\Omega\to \R^{n\times n}$ be uniformly elliptic with ellipticity constant $\lambda$ and let $\B$ be a vector field satisfying $\B\in \mathbf{\dot F^2}_\mu(\Omega)$, for some $0<\mu<\infty$, and $\Div\B\le 0$ in the weak sense. Denote by $L$ the operator given by $L=-\Div(A\nabla)+\B\cdot\nabla$. Let $Y\in \Omega$ and denote $\omega_L^Y=\omega^Y$ and $G_L^Y=G^Y$.
Then, for all $X\in \Omega$ such that $|X-Y|\ge \delta(Y)/4$ we have
\begin{equation}\label{eqn:cfms}
\frac{G^Y(X)}{\delta(Y)}\le C\, \frac{\omega^X(\Delta_Y)}{\sigma(\Delta_Y)},
\end{equation}
where $C$ is a constant depending only on $n$, $\mu$, $\lambda$, and the geometric constants of $\Omega$. Recall that $\sigma$ is the $(n-1)$-dimensional Hausdorff measure.

Moreover, if $\Div\B=0$, then, for all $X\in \Omega$ such that $|X-Y|\ge 100K\delta(Y)$ we have 
\begin{equation}\label{eqn:cfms:1}
\frac{G^Y(X)}{\delta(Y)}\ge C^{-1}\, \frac{\omega^X(\Delta_Y)}{\sigma(\Delta_Y)},
\end{equation}
where $K$ is the constant in the statement of Corollary~\ref{cor:carl}, and $C$ is a constant depending only on $n$, $\mu$, $\lambda$, and the geometric constants of $\Omega$. 
\end{lem}

\begin{proof}
As in \cite[Lemma 3.45]{Hof26}, we follow \cite[p. 10]{Ken94}. Let $z\in\partial\Omega$ be such that $|z-Y|=\delta(Y)$. Then, $\Delta(z,9\delta(Y))\subset \Delta_Y$. If $10\delta(Y)>\diam(\Omega)$, then $\Delta_Y=\partial\Omega$, and~\eqref{eqn:cfms} follows directly from $\omega^X(\partial\Omega)=1$,~\eqref{eqn:green:sup}, and Ahlfors regularity. Thus, we may assume that $10\delta(Y)\le\diam(\Omega)$. 
For any $X: |X-Y|=\delta(Y)/4$ we also have $X\in D(z,5\delta(Y)/4)$. Therefore, by Lemma~\ref{lem:bourgain}, we have 
\[
\omega^X(\Delta_Y)\ge\omega^X(\Delta(z,5\delta(Y)))\ge c.
\]
By~\eqref{eqn:green:sup} and Ahlfors regularity we get
\[
\frac{\omega^X(\Delta_Y)}{\sigma(\Delta_Y)}\ge\frac{\omega^X(\Delta(z,5\delta(Y)))}{\sigma(\Delta_Y)} \ge C^{-1}\delta(Y)^{1-n}\ge C^{-1}\frac{G^Y(X)}{\delta(Y)}, \qquad |X-Y|=\delta(Y)/4.
\]
Take now $f\in C^\infty_c(B(z,9\delta(Y)))$ with $0\le f\le 1$ and $f=1$ in $B(z,5\delta(Y))$, and let $u_f$ be the solution to~\eqref{eqn:dir}. For any $X\in \Omega$ we have
\[
\omega^X(\Delta_Y)\ge u_f(X)\ge \omega^X(\Delta(z,5\delta(Y))).
\]
Therefore, an application of Lemma~\ref{lem:max:princ} in $\Omega\setminus B(Y,\delta(Y)/4)$ to the function $\frac{u_f}{\sigma(\Delta_Y)}-C^{-1}\frac{G^Y}{\delta(Y)}$ yields~\eqref{eqn:cfms}.

We now turn to the proof of~\eqref{eqn:cfms:1}. Let $\Phi\in C^\infty_c(B(z,12\delta(Y)))$ be such that $0\le\Phi\le 1$, $\Phi=1$ in $B(z,11\delta(Y))$, $|\nabla\Phi|\le 2/\delta(Y)$, and $|\nabla^2\Phi|\le 2/\delta(Y)^2$. For any $X:|X-Y|\ge 100K\delta(Y)$ we have $\Phi(X)=0$. Thus, since $\Delta_Y\subset \Delta(z,11\delta(Y))$, by~\eqref{eqn:rep:h} and~\eqref{eqn:rep:1}, we have
\begin{equation}\label{eqn:step:cfms}
\omega^X(\Delta_Y)\le u_\Phi(X)=-\int_{D(z,12\delta(Y))}A^T\nabla G^X_{L^*}\nabla \Phi-\int_{D(z,12\delta(Y))}\B\cdot\nabla \Phi G^X_{L^*}.
\end{equation}
Using H\"older's inequality, $|\nabla\Phi|\le 2/\delta(Y)$, $|\nabla^2\Phi|\le 2/\delta(Y)^2$, the uniform boundedness of $A$, and the fact that $\B\in \mathbf{\dot F^2}_\mu(\Omega)$, we estimate the right-hand side of~\eqref{eqn:step:cfms} as
\begin{align*}
-\int_{D(z,12\delta(Y))}A^T\nabla& G^X_{L^*}\nabla \Phi-\int_{D(z,12\delta(Y))}\B\cdot\nabla \Phi G^X_{L^*}\le C \delta(Y)^{\frac{n-2}{2}} \Bigl(\int_{D(z,12\delta(Y))}|\nabla G^X_{L^*}|^2\Bigr)^{1/2}\\
&+ C\delta(Y)^{\frac{n}{2}}\Bigl(\int_{\Omega}|\B|^2|\nabla\Phi G^X_{L^*}|^2\Bigr)^{1/2}\le C \delta(Y)^{\frac{n-2}{2}} \Bigl(\int_{D(z,12\delta(Y))}|\nabla G^X_{L^*}|^2\Bigr)^{1/2}\\
&+C\delta(Y)^{\frac{n}{2}}\Bigl(\int_{D(z,12\delta(Y))}|\nabla^2\Phi|^2| G^X_{L^*}|^2\Bigr)^{1/2}\\
&\le C \delta(Y)^{\frac{n-2}{2}} \Bigl(\int_{D(z,12\delta(Y))}|\nabla G^X_{L^*}|^2\Bigr)^{1/2}+C\delta(Y)^{\frac{n-4}{2}}\Bigl(\int_{D(z,12\delta(Y))}| G^X_{L^*}|^2\Bigr)^{1/2}.
\end{align*}
Notice that $X\notin B(z,13\delta(Y))$. Therefore, by~\eqref{eqn:green}, $G^X_{L^*}$ is a solution with zero trace in $B(z,13\delta(Y))$. By~\eqref{c_ineq}, applied with $\eta\in C^\infty_c(B(z,13\delta(Y)))$, $0\le\eta\le1$, $\eta=1$ in $B(z,12\delta(Y))$, and $|\nabla\eta|\le 2/\delta(Y)$, we can estimate the first term on the right-hand side of the last inequality, getting
\[
-\int_{D(z,12\delta(Y))}A^T\nabla G^X_{L^*}\nabla \Phi-\int_{D(z,12\delta(Y))}\B\cdot\nabla \Phi G^X_{L^*}\le C\delta(Y)^{\frac{n-4}{2}}\Bigl(\int_{D(z,13\delta(Y))}| G^X_{L^*}|^2\Bigr)^{1/2}.
\]
We now estimate the right-hand side using~\eqref{eqn:carl}, with $u=G^X_{L^*}$, and we plug the result in~\eqref{eqn:step:cfms}, obtaining
\[
\omega^X(\Delta_Y)\le C\delta(Y)^{n-2}G^X_{L^*}(Y_\Delta),
\]
where $Y_\Delta$ is a corkscrew point for $z$ at scale $13\delta(Y)$. By Ahlfors regularity $\sigma(\Delta_Y)\approx\delta(Y)^{n-1}$, and by the interior Harnack inequality~\eqref{eqn:h}, together with the interior Harnack chain condition, we conclude $G^X_{L^*}(Y_\Delta)\approx G^X_{L^*}(Y)$. Finally,~\eqref{eqn:sym} gives $G^X_{L^*}(Y)=G^Y_{L}(X)$, ending the proof of~\eqref{eqn:cfms:1}.
\end{proof}

As a corollary, we obtain that the harmonic measure is doubling.
\begin{cor}\label{cor:doub}
Let $\Omega\subset\R^n$ be a one-sided chord-arc domain with unbounded complement.
Let $A:\Omega\to \R^{n\times n}$ be uniformly elliptic with ellipticity constant $\lambda$ and let $\B$ be a vector field satisfying $\B\in \mathbf{\dot F^2}_\mu(\Omega)$, for some $0<\mu<\infty$, and $\Div\B= 0$ in the weak sense.
Let $x\in\partial\Omega$, $0<r<\infty$, and $X\in\Omega\setminus B(x,4Kr)$, where $K>1$ is the scaling constant defined in Lemma~\ref{lem:tent}. Then,
\begin{equation}\label{eqn:doub}
\omega^X(\Delta(x,2r))\le C\omega^X(\Delta(x,r)),
\end{equation}
where $C$ is a constant depending only on $n$, $\mu$, $\lambda$, and the geometric constants of $\Omega$. 
\end{cor}

\begin{proof}
Let $M\gg1$ be sufficiently large, depending on the geometric constants of $\Omega$ and on $n$, and set $s:=r/M$. By Ahlfors regularity, we can cover $\Delta(x,2r)$ with $N$ surface balls $\{\Delta(x_i,s)\}_{i=1,\dots,N}$ centered in $\Delta(x,2r)$, where $N$ depends only on the geometric constants of $\Omega$ and on $n$. For any $1\le i\le N$, let $Y_i$ be a corkscrew point for $x_i$ at scale $13s$. Taking $M$ sufficiently large, the argument used in the proof of~\eqref{eqn:cfms:1}, with a cutoff adapted to $\Delta(x_i,s)$, gives
\begin{equation}\label{eqn:step:doub}
\omega^X(\Delta(x,2r))
\le C\sum_{i=1}^N\omega^X(\Delta(x_i,s))
\le Cs^{n-2}\sum_{i=1}^N G^X_{L^*}(Y_i).
\end{equation}
Let $L=-\Div(A\nabla)+\B\cdot\nabla$ and let $Y$ be a corkscrew point for $x$ at scale $r/M$. Since $\delta(Y_i)\approx r$, $\delta(Y)\approx r$, and $|Y_i-Y|\lesssim r$, for $M$ big enough, the Harnack chain condition and~\eqref{eqn:h} give
\[
G^X_{L^*}(Y_i)\le C G^X_{L^*}(Y)=CG^Y(X).
\]
Since $M$ and $N$ are fixed,~\eqref{eqn:step:doub} yields
\[
\omega^X(\Delta(x,2r))\le Cr^{n-2}G^Y(X).
\]
By~\eqref{eqn:cfms} we conclude
\[
\omega^X(\Delta(x,2r))\le C\omega^X(\Delta_Y)\le C\omega^X(\Delta(x,r)),
\]
where we used that our choice of $Y$ guarantees that $\Delta_Y\subset\Delta(x,r)$.
\end{proof}

We now show a comparison principle for nonnegative solutions with zero trace.
\begin{lem}\label{lem:comp:princ}
Let $\Omega\subset\R^n$ be a one-sided chord-arc domain with unbounded complement.
Let $A:\Omega\to \R^{n\times n}$ be uniformly elliptic with ellipticity constant $\lambda$ and let $\B$ be a vector field satisfying $\B\in \mathbf{\dot F^2}_\mu(\Omega)$, for some $0<\mu<\infty$, and $\Div\B=0$ in the weak sense.
Let $x\in\partial\Omega$, $0<r\le\diam(\Omega)$, and let $X_r$ be a corkscrew point for $x$ at scale $r$. There exists a constant $C_0>1$, depending only on $n$ and on the geometric constants of $\Omega$, such that the following holds. Let $u, v\in W^{1,2}(T(x,2C_0r))$ be nonnegative and not identically zero solutions in $T(x,2C_0r)$ (see Lemma~\ref{lem:tent}) with $\Tr u=\Tr v=0$ in $\partial T(x,2C_0r)\cap\partial\Omega$. Then, 
\begin{equation}\label{eqn:comp:princ}
C^{-1}\frac{u(X_r)}{v(X_r)}\le\frac{u(X)}{v(X)}\le C\frac{u(X_r)}{v(X_r)},\quad \forall X\in B(x,r)\cap\Omega,
\end{equation}
where $C$ is a constant depending only on $n$, $\mu$, $\lambda$, and the geometric constants of $\Omega$. 
\end{lem}
\begin{proof}
We adapt the arguments in \cite[Lemma 1.3.7]{Ken94} to our setting. Let $C_0>1$ be a number to be chosen momentarily. First, by symmetry, it is sufficient to prove the upper bound in~\eqref{eqn:comp:princ}. Since $u$ and $v$ are not identically zero, then $\Omega\setminus T(x,2C_0r)\ne\emptyset$. Indeed, if $T(x,2C_0r)=\Omega$, by Lemma~\ref{lem:existence} we would have $u=v=0$.  Given this, by connecting through a Harnack chain $X_r$ to a corkscrew point at a scale proportional to $r$ in $\Omega\setminus T(x,5C_0r/3)$, there exists $\epsilon>0$, depending only on the geometric constants of $\Omega$ and on $C_0$, such that
\[
\Gamma_{2\epsilon}:=\partial T(x,3C_0r/2)\cap\Omega\cap\{X\in\Omega:\delta(X)\ge 2\epsilon r\}\ne \emptyset.
\]
Let $\xi\in \Gamma_{2\epsilon}$ and notice that $B(\xi,\epsilon r)\cap\partial T(x,3C_0r/2)\subset\Gamma_{\epsilon}$, where $\Gamma_\epsilon=\partial T(x,3C_0r/2)\cap\Omega\cap\{X\in\Omega:\delta(X)\ge \epsilon r\}$. Set $\Gamma=\partial T(x,3C_0r/2)\cap\overline{\Omega}$, and notice that $\Gamma\subset B(\xi,M\epsilon r)$, where $M$ depends only on the geometric constants of $\Omega$ and on $C_0$. Since $D(x,15C_0r/8)\subset T(x,3C_0r/2)$, we have $\dist(X,\Gamma)\ge 7C_0r/8$ for every $X\in B(x,r)\cap\Omega$. Therefore, we choose $C_0$, depending only on the geometric constants of $\Omega$ and on $n$, sufficiently big so that the covering argument used in the proof of Corollary~\ref{cor:doub}, applied in the one-sided chord-arc domain $T(x,3C_0r/2)$, gives
\begin{equation}\label{eqn:step:comp:princ}
\omega_{T(x,3C_0r/2)}^X(\Gamma)\le C\omega_{T(x,3C_0r/2)}^X(\Gamma_\epsilon),\quad\forall X\in B(x,r)\cap\Omega.
\end{equation}
Here $\omega_{T(x,3C_0r/2)}$ is the harmonic measure for the operator $L=-\Div(A\nabla)+\B\cdot\nabla$ in $T(x,3C_0r/2)$, and $C$ depends only on $n$, $\mu$, $\lambda$, and the geometric constants of $\Omega$.

Using~\eqref{eqn:carl} in small balls centered in $\Gamma$ and using the Harnack chain condition in $T(x,3C_0r/2)$ to connect the corkscrew points of such balls to $X_r$, we estimate, by~\eqref{eqn:h},
\[
u(X)\le Cu(X_r), \quad X\in \partial T(x,3C_0r/2),
\]
where we also used that $u=0$ on $\partial T(x,3C_0r/2)\cap\partial\Omega$.
Let now $f_j\in C^\infty_c(\R^n)$ with $0\le f_j\le 1$, $f_j=1$ in $\Gamma$, and $f_j=0$ outside of $\{X:\dist(X,\Gamma)<1/j\}$. Let $u_{f_j}$ be the solution to $Lu=0$ in $T(x,3C_0r/2)$ with $\Tr u_{f_j}=f_j|_{\partial T(x,3C_0r/2)}$ on $\partial T(x,3C_0r/2)$. By the maximum principle (see Lemma~\ref{lem:max:princ}) applied in the domain $T(x,3C_0r/2)$, we get
\[
u(X)\le Cu(X_r)u_{f_j}(X),\quad X\in T(x,3C_0r/2).
\]
Sending $j\to\infty$ we can write
\begin{equation}\label{eqn:step:1:comp:princ}
u(X)\le Cu(X_r)\omega^X_{T(x,3C_0r/2)}(\Gamma),\quad X\in T(x,3C_0r/2).
\end{equation}
We now move to the estimate of $v$ from below. By~\eqref{eqn:h} and by the Harnack chain condition we can write
\[
v(X)\ge C^{-1}v(X_r), \quad X\in\Gamma_\epsilon.
\]
Taking $g_j\in C^\infty_c(\R^n)$ such that $0\le g_j\le 1$, $g_j=1$ in $\{X\in \Gamma_\epsilon:\delta(X)\ge\epsilon r(1+2^{-j})\}$, and $g_j=0$ in $\partial T(x,3C_0r/2)\setminus \Gamma_\epsilon$, by the maximum principle we have 
\[
v(X)\ge C^{-1}v(X_r) u_{g_j}(X),\quad X\in T(x,3C_0r/2).
\]
Sending $j\to\infty$ we can write
\begin{equation}\label{eqn:step:2:comp:princ}
v(X)\ge C^{-1}v(X_r)\omega^X_{T(x,3C_0r/2)}(\Gamma_\epsilon), \quad X\in T(x,3C_0r/2).
\end{equation}
Finally, combining~\eqref{eqn:step:2:comp:princ} and~\eqref{eqn:step:1:comp:princ} with~\eqref{eqn:step:comp:princ} when $X\in B(x,r)\cap\Omega$, we end the proof of~\eqref{eqn:comp:princ}.
\end{proof}

The last result of this section is the change of pole for harmonic measure. 
\begin{cor}\label{cor:change}
Let $\Omega\subset\R^n$ be a one-sided chord-arc domain with unbounded complement.
Let $A:\Omega\to \R^{n\times n}$ be uniformly elliptic with ellipticity constant $\lambda$ and let $\B$ be a vector field satisfying $\B\in \mathbf{\dot F^2}_\mu(\Omega)$, for some $0<\mu<\infty$, and $\Div\B=0$ in the weak sense.
Let $x\in\partial\Omega$, $0<r\le\diam(\Omega)$, and let $X_r$ be a corkscrew point for $x$ at scale $r$. Then, for every Borel set $E\subset \Delta(x,r)$, we have 
\begin{equation}\label{eqn:change}
C^{-1}\omega^{X_r}(E)\le\frac{\omega^X(E)}{\omega^X(\Delta(x,r))}\le C\omega^{X_r}(E),\quad \forall X\in\Omega:|X-x|\ge 2r.
\end{equation}
\end{cor}
\begin{proof}
By Borel regularity of harmonic measure we may assume that $E$ is bounded and open in $\partial\Omega$. Let $j\in\N$, and let $f_j, g_j\in C^\infty_c(\R^n)$ be such that $0\le f_j,g_j\le 1$, $f_j=1$ in $\{z\in E: \dist(z,\partial\Omega\setminus E)>1/j\}$, $g_j=1$ in $\Delta(x,r(1-2^{-j}))$, $f_j=0$ in $\partial\Omega\setminus E$, and $g_j=0$ in $\partial\Omega\setminus \Delta(x,r)$. Note that $\Tr u_{f_j}=\Tr u_{g_j}=0$ on $\partial\Omega\setminus \Delta(x,r)$. For any $z\in\partial\Omega\cap B(x,3r)\setminus B(x,3r/2)$, we apply~\eqref{eqn:comp:princ} with center $z$ and scale $r/M$, where $M>1$ is a big enough number depending only on the geometric constants of $\Omega$. We obtain 
\[
C^{-1}\frac{u_{f_j}(Z_{r/M})}{u_{g_j}(Z_{r/M})}\le\frac{u_{f_j}(X)}{u_{g_j}(X)}\le C \frac{u_{f_j}(Z_{r/M})}{u_{g_j}(Z_{r/M})}, \quad X\in B(z,r/M)\cap\Omega,
\]
where $Z_{r/M}$ is a corkscrew point for $z$ at scale $r/M$. Connecting $Z_{r/M}$ and $X_r$ with a Harnack chain and using~\eqref{eqn:h} we can write
\begin{equation}\label{eqn:step:change}
C^{-1}\frac{u_{f_j}(X_r)}{u_{g_j}(X_r)}\le\frac{u_{f_j}(X)}{u_{g_j}(X)}\le C \frac{u_{f_j}(X_r)}{u_{g_j}(X_r)}, \quad X\in B(z,r/M)\cap\Omega.
\end{equation}
Moreover, for any $X\in\partial B(x,2r)\cap\Omega$, either $X\in B(z,r/M)$ for some $z\in\partial\Omega\cap B(x,3r)\setminus B(x,3r/2)$, or $\delta(X)\ge r/M$. In the first case,~\eqref{eqn:step:change} holds for $X$, while in the second case, we can connect $X$ and $X_r$ with a Harnack chain and, using once again~\eqref{eqn:h}, we obtain that~\eqref{eqn:step:change} holds for all $X\in\partial B(x,2r)\cap\Omega$. Consider the solutions $w_1:=u_{f_j}-Cu_{g_j}\frac{u_{f_j}(X_r)}{u_{g_j}(X_r)}$ and $w_2:=u_{f_j}-C^{-1}u_{g_j}\frac{u_{f_j}(X_r)}{u_{g_j}(X_r)}$. Both functions have zero trace in $\partial\Omega\setminus B(x,2r)$, and $w_1\le 0$ in $\partial B(x,2r)\cap\Omega$, $w_2\ge 0$ in $\partial B(x,2r)\cap\Omega$. The maximum principle, applied in $\Omega\setminus B(x,2r)$, yields $w_1\le 0$ everywhere and $w_2\ge 0$ everywhere. That is,~\eqref{eqn:step:change} holds for any $X:|X-x|\ge 2r$. Sending $j\to\infty$ we obtain
\begin{equation}\label{eqn:step:1:change}
C^{-1}\frac{\omega^{X_r}(E)}{\omega^{X_r}(\Delta(x,r))}\le \frac{\omega^{X}(E)}{\omega^{X}(\Delta(x,r))}\le C\frac{\omega^{X_r}(E)}{\omega^{X_r}(\Delta(x,r))}, \quad |X-x|\ge 2r.
\end{equation}
Using~\eqref{eqn:bourgain} and connecting $X_r$ with a corkscrew point for $x$ at scale $r/4$ we find $\omega^{X_r}(\Delta(x,r))\ge c$. Therefore,~\eqref{eqn:step:1:change}, together with $\omega^{X_r}(\Delta(x,r))\le 1$, gives~\eqref{eqn:change}, ending the proof.
\end{proof}

\begin{rmk}
In \cite{Hof26}, the author observed that it was not known whether
the change-of-pole formula holds under the smallness assumption
$|\B|(X)\le \epsilon/\delta(X)$, for all $X\in\Omega$, without any
condition on $\Div\B$. The obstruction to the Green-function
argument described after \cite[Lemma~3.50]{Hof26} was that a
boundary Harnack principle for the adjoint operator $L^*$ was
not available under these assumptions.

Although $L^*$ is of the same form as $L$ in our divergence-free
setting, our proof of Corollary~\ref{cor:change} applies
Lemma~\ref{lem:comp:princ} to $L$-harmonic approximations of
harmonic measure, avoiding the use of boundary Harnack for
adjoint solutions. This strategy has precedents; see
\cite[Lemma~11.16]{DavFM21}. The same argument therefore also
yields the change-of-pole formula in the small constant setting
of \cite{Hof26}, without any condition on $\Div\B$. We record
this consequence in the following corollary.
\end{rmk}

\begin{cor}
Let $\Omega\subset\R^n$ be a one-sided chord-arc domain with unbounded complement.
There exists $\epsilon>0$, depending only on $n$, $\lambda$, and the geometric constants of $\Omega$, such that the following holds. Let $A:\Omega\to \R^{n\times n}$ be uniformly elliptic with ellipticity constant $\lambda$ and let $\B$ be a vector field satisfying 
\[
|\B|(X)\le \frac{\epsilon}{\delta(X)},\quad\forall X\in\Omega.
\]
Let $x\in\partial\Omega$, $0<r<\diam(\Omega)$, and let $X_r$ be a corkscrew point for $x$ at scale $r$. Then, for every Borel set $E\subset \Delta(x,r)$, we have 
\begin{equation*}
C^{-1}\omega^{X_r}(E)\le\frac{\omega^X(E)}{\omega^X(\Delta(x,r))}\le C\omega^{X_r}(E),\quad \forall X\in\Omega:|X-x|\ge 2r,
\end{equation*}
where $C$ is a constant depending only on $n$, $\lambda$, and the geometric constants of $\Omega$.
\end{cor}

\section{Dimension drop in the upper half space $\R^3_+$}\label{sec:9}
In this section, for each $\epsilon>0$, we will construct a drift $\B$ on $\R^3_+:=\{(x_1,x_2,x_3)\in\R^3:x_3>0\}$ that satisfies the smallness assumption
\begin{equation}\label{eqn:smalldrift}
|\B(X)|\le \frac{\epsilon}{\delta(X)},\quad X\in\R^3_+,
\end{equation}
and the divergence-free condition
\begin{equation}\label{eqn:divfree}
\Div(\B)=0.
\end{equation}
We consider the operator 
\begin{equation}\label{eqn:op:drop}
L=-\Delta+\B\cdot\nabla
\end{equation}
in the domain $\Omega=\R^3_+$. We show an example of such a drift $\B$ for which the harmonic measure $\omega_{L}$ exhibits a dimension drop. The nature of our example is two-dimensional. The reason why we present it in three dimensions is that we developed the elliptic theory that we will use in this section only in dimensions larger than or equal to three (see Remark~\ref{rmk:2dim}).

\begin{thm}\label{thm:dim:drop}
Let $\epsilon>0$. There exists $q_0\in\N$ and $\gamma>0$ such that the following hold. Define, for $q\ge q_0$,
\begin{align}
 \Phi(x,y,z)
   &=\epsilon\sum_{j=0}^{\infty}e^{-q^jz}\cos(q^jx),\label{eqn:potential}\\
 \B(x,y,z)=-\nabla\Phi(x,y,z)
   &=\epsilon\sum_{j=0}^{\infty}q^je^{-q^jz}
           \bigl(\sin(q^jx),0,\cos(q^jx)\bigr).\label{eqn:field}
\end{align}
Then, $\B$ satisfies~\eqref{eqn:smalldrift} and~\eqref{eqn:divfree}. Moreover, $\B\in C^\infty(\R^3_+;\R^3)$.
Let $L$ be as in~\eqref{eqn:op:drop}, and $\omega_{L}=:\omega$ be the harmonic measure given by Lemma~\ref{lem:hm}.
Then, 
\[
\dimH\omega<2=\dimH\partial\R^3_+.
\]
More explicitly, set
\begin{equation}\label{eqn:carrier}
 S_m(x)=\sum_{j=0}^{m-1}\cos(q^jx),\qquad
 E=\left\{x\in\R:\liminf_{m\to\infty}\frac{S_m(x)}m
                        \ge\gamma\right\},\qquad F=E\times\R.
\end{equation}
Then, for every $X\in\Omega$,
\begin{equation}\label{eq:mainconclusion}
 \omega^X(F)=1,\qquad \dimH F\le2-\frac{\gamma^2}{32\log q}<2.
\end{equation}
\end{thm}
\begin{rmk}
A straightforward consequence of the Harnack inequality and of the definition of Harmonic measure is that $\omega^X$ and $\omega^Y$ have the same null sets for any $X\ne Y$. Therefore, we can prove~\eqref{eq:mainconclusion} for a fixed pole $X_0\in \R^3_+$
\end{rmk}
The current section is devoted to the proof of Theorem~\ref{thm:dim:drop}.

We begin with a Lemma on the basic properties of $\B$ and $\Phi$. 
\begin{lem}\label{lem:drop:1}
Let $\epsilon>0$, $q\ge q_0>>1$, $m\in\N$, and let $\B$, $\Phi$, and $S_m$ as in~\eqref{eqn:field} and~\eqref{eqn:potential}. Then, $\Div (\B)=0$, and the following size estimates hold. 
\begin{equation}\label{eqn:size:field}
|\B|(x,y,z)\le \frac{2\epsilon}{z},\quad\forall (x,y,z)\in \R^3_+,
\end{equation}
\begin{equation}\label{eqn:size:field:1}
|\B|(x,y,z)\ge \frac{\epsilon}{10} q^{m} \quad\forall (x,y,z)\in \R^3_+:q^{-m}/2\le z\le 2q^{-m},
\end{equation}
\begin{equation}\label{eqn:size:sum}
\Bigl|\Phi(x,y,z)-\epsilon S_m(x)\Bigr|\le C\epsilon,\quad \forall (x,y,z)\in \R^3_+:q^{-m}\le z<q^{-m+1}, 
\end{equation}
where $C$ is a universal constant.
Finally, for any $(x,y),(x',y')\in\R^2$ we have
\begin{equation}\label{eqn:size:horizontal}
|\Phi(x,y,q^{-m})-\Phi(x',y',q^{-m})|\le 2\epsilon +4\epsilon\log(q^m|(x,y)-(x',y')|).
\end{equation}
\end{lem}
\begin{proof}
For each $j\ge 0$ the function $\Phi_j(x, y, z):=e^{-q^jz}\cos(q^jx)$ satisfies $\Delta \Phi_j=0$ in $\R^3$. The partial sums $\sum_{0\le j\le m}\Phi_j$ converge uniformly on compact subsets of $\R^3_+$; therefore, the limit $\Phi$ is harmonic as well, and $\B=-\nabla \Phi$ is well defined, given by~\eqref{eqn:field}, and smooth in $\R^3_+$. In particular, $\Div\B=-\Delta\Phi=0$. Take now $(x,y,z)\in R^3_+$, and notice that, for any $j\ge 0$, 
\[
q^je^{-q^jz}\le q^j\fint_{q^{j-1}}^{q^j}e^{-tz}\,dt=\frac{q}{q-1}\int_{q^{j-1}}^{q^j}e^{-tz}\,dt.
\] 
Therefore, we estimate
\begin{equation}\label{eqn:size:step}
|\B|(x,y,z)\le \epsilon\sum_{j\ge 0}q^je^{-q^jz}\le \frac{q}{q-1}\frac{\epsilon}{z}\le \frac{2\epsilon}{z}.
\end{equation}

Assume now that $q^{-m}/2\le z\le 2q^{-m}$. Then we have, by triangle inequality,
\begin{align*}
|\B(x,y,z)|&\ge \epsilon q^m\Bigl(e^{-2}-\sum_{0\le j\le m-1}q^{j-m}-\sum_{j\ge m+1}q^{j-m}e^{-q^{j-m}/2}\Bigr)\\
&\ge  \epsilon q^m\Bigl(e^{-2}-\frac{1}{q-1}-\sum_{\ell\ge 1}q^\ell e^{-q^\ell /2}\Bigr)\\
&\ge \epsilon q^m/10,
\end{align*}
where the last inequality holds for $q\ge q_0$ sufficiently large.

Let now $q^{-m}\le z<q^{-(m-1)}$. We have, by the triangle inequality, 
\begin{align*}
|\Phi(x,y,z)-\epsilon S_m(z)|&\le \epsilon\sum_{j\ge m}e^{-q^{j-m}}+\epsilon \sum_{0\le j\le m-1}1-e^{-q^{j-m+1}}\\
&\le \epsilon\sum_{j\ge 0}e^{-q^j}+\epsilon\sum_{i=0}^{m-1}q^{-i}\le \epsilon C.
\end{align*}
Finally, let $(x,y),(x',y')\in\R^2$, 	assume that $h=|X-X'|\ge q^{-m}$, and set $X=(x,y,q^{-m})$, $X'=(x',y',q^{-m})$, $Z=(x,y,h)$, $Z'=(x',y',h)$. We estimate $|\Phi(X)-\Phi(X')|$ by using~\eqref{eqn:size:field:1} on the three segments $[X,Z]$, $[Z,Z']$, and $[Z',X']$. The horizontal segment has length equal to its height, so $|\Phi(Z)-\Phi(Z')|\le 2\epsilon |Z-Z'|/h\le 2\epsilon$. On the vertical segment we have logarithmic loss, that is, 
\[
|\Phi(X)-\Phi(Z)|+|\Phi(Z')-\Phi(X')|\le 4\epsilon\int_{q^{-m}}^h\frac 1z \,dz=4\epsilon \log(q^mh)|X-X'|,
\]
completing the proof of~\eqref{eqn:size:horizontal}
\end{proof}
\begin{rmk}
We need $q_0$ large in Theorem~\ref{thm:dim:drop} only to guarantee that $\B$ satisfies the properties~\eqref{eqn:size:field} and~\eqref{eqn:size:field:1}. Looking at the proof of Lemma~\ref{lem:drop:1}, it is clear that the choice $q_0=2$ provides~\eqref{eqn:size:field}. However, for~\eqref{eqn:size:field:1} to hold, we need stronger scale separation. A computation shows that $q_0=32$ is sufficient. 
\end{rmk}
To prove Theorem~\ref{thm:dim:drop}, we will estimate from above the harmonic measure of regions where $\frac{S_m}{m}$ is small. As an intermediate step, we will express the harmonic measure $\omega$ in terms of the harmonic measure relative to the operator in terms of L in the domain $H_m:=\{(x,y,z)\in\R^3_+: z>q^{-m}\}$. The next Lemma allows us to do so. This is a classical result, but we include its proof for the sake of completeness.
\begin{lem}\label{lem:comp:h}
Let $m\in\N$, $X_0=(x_0,y_0,z_0)$ with $z_0>q^{-m}$, and let $\omega_m$ be the harmonic measure relative to the operator $L$ in the domain $H_m=\{(x,y,z)\in\R^3_+: z>q^{-m}\}$. For all $D\subset \R^2$  Borel sets, we have 
\begin{equation}\label{eqn:comp:h}
\omega^{X_0}(D)=\int_{\{z=q^{-m}\}}\omega^P(D)\,d\omega_m(P).
\end{equation}
In particular, for any bounded Borel function $g:\R^2\to \R$ we have 
\begin{equation}\label{eqn:comp:h:1}
\int_{\R^2}g\,d\omega^{X_0}=\int_{\{z=q^{-m}\}}\int_{\R^2}g\,d\omega^P\,d\omega_m(P).
\end{equation}
\end{lem}
\begin{proof}
A straightforward consequence of the definition of the harmonic measure is that $P\to\omega^P(D)$ is continuous in $\Omega$, as can be seen by approximating $\omega$ by solutions. Therefore, the integral on the right-hand side of~\eqref{eqn:comp:h} is well defined. Up to approximating $D$ with open bounded sets, we can also assume that $D$ is bounded and open. For $i\in\N$, let $f_i\in C^\infty_c(\R^2)$ be such that $f_i=1$ in $\{(x,y)\in\R^2:\dist((x,y),\R^2\setminus D)>1/i\}$ and $f_i=0$ outside of $D$. Let $u_{f_i}$ be the unique solution, given by Lemma~\ref{lem:existence}, to $Lu=0$ in $\R^3_+$ with $\Tr u=f_i$ on $\R^2$. Then $u_{f_i}$ is the unique solution to $Lu=0$ in $H_m$ with $\Tr u_{f_i}=u_{f_i}|_{\{z=q^{-m}\}}$ on $\partial H_m$, and by the definition of $\omega_m$ we have
\[
u_{f_i}(X_0)=\int_{\{z=q^{-m}\}}u_{f_i}(P)\,d\omega^{X_0}_m.
\]
Taking the limit as $i\to\infty$ yields~\eqref{eqn:comp:h}.
\end{proof}
We use $\omega_m$ because, in the strip $\{q^{-m}\le z<q^{-m+1}\}$, we can isolate more effectively the effect that the first $m$ scales of the potential $-\Phi$ have on the behavior of the harmonic measure. To do so, we will use the auxiliary function
\begin{equation}\label{eqn:barrier}
v:=e^{-\Phi/2}.
\end{equation}
Having in mind the connection between $\Phi$ and $S_m$ established in~\eqref{eqn:size:sum}, fixing $X_0\in\R^3_+$ and $\gamma>0$ we can estimate, for $m> C/\gamma$, 
\begin{equation}\label{eqn:drop:base}
\omega_m^{X_0}\Bigl(\{S_m\le \gamma m\}\times\R\times{q^{-m}}\Bigr)\le e^{\epsilon\gamma m} \int_{\R^2}v(x,y,q^{-m}) \,d\omega^{X_0}_m(x,y),
\end{equation}
where the integral is well defined since $\Phi$, and hence $v$, is uniformly bounded on $H_m$.
Our first aim is to prove that the $L$-harmonic extension of $v$ above height $q^{-m}$ is exponentially small in $ m$.
Set 
\begin{equation}\label{eqn:eqn:cut:m}
V_{m}(X):=\int_{\{z=q^{-m}\}}v(x,y,q^{-m})\,d\omega^X_m.
\end{equation}
 The advantage of using $v$ is that such function is $L$-superharmonic, with 
\begin{equation}\label{eqn:Lv}
L v=\frac{v|\nabla\Phi|^2}{4}=\frac{v|\B|^2}4.
\end{equation}
Since at height close to $q^{-m}$, $v$ is bounded from above and below by positive constants, the size estimate~\eqref{eqn:size:field:1} and~\eqref{eqn:Lv} suggest using the maximum principle to compare $v-V_m$ and a quadratic polynomial. 
We will make this idea precise in the proof of the following lemma, obtaining the desired exponential decay.
\begin{lem}\label{lem:drop:key}
There exists $\rho\in (0,1)$, depending only on $\epsilon$, such that the following holds.
Let $m\ge 1$, $\gamma>0$, and $X_0=(x_0,y_0,q^{-(m-1)})$. Let $C$ be the constant in~\eqref{eqn:size:sum}. Then,
\begin{equation}\label{eqn:exp:decr}
V_{m}(X_0)\le \rho v(X_0).
\end{equation}
Moreover, if $X_0=(x_0,y_0,z_0)$ with $z_0\ge 2$,
\begin{equation}\label{eqn:exp:decr:1}
V_{ m}(X_0)\le \rho^{m} v(X_0),
\end{equation}
and, if $m>C/\gamma$,
\begin{equation}\label{eqn:exp:decr:2}
\omega_m^{X_0}\Bigl(\{S_m\le \gamma m\}\times\R\times\{q^{-m}\}\Bigr)\le (e^{\epsilon\gamma}\rho)^mv(X_0).
\end{equation}
\end{lem}
\begin{proof}
Fix $X_0$ as in the statement. 
For each $k\in\N$ set, for $X\in H_{m+1}$, 
\[
V_{{m}}^{(k)}(X)=\int_{\{z=q^{-m}\}}\eta_{k}v(x,y,q^{-m})\,d\omega^X_{m},
\]
where $0\le \eta_k\le 1$ is a smooth function supported in $B(0,2k)$, with $\eta_k=1$ in $B(0,k)$. In such a way, we have $V_{m}^{(k)}\in Y^{1,2}(H_{m})\cap C(\overline{H_{m}})$, $V_{m}^{(k)}|_{\{z=q^{-m}\}}=\eta_kv(\cdot,\cdot,q^{-m})$, and $L V^{(k)}_{m}=0$ in $H_{m}$. 
Since the boundary data $\eta_kv(\cdot,\cdot,q^{-m})$ is compactly supported, by~\eqref{eqn:hold:c} and~\eqref{eqn:sup:int} applied with $p=6$, we get 
\begin{equation}\label{eqn:zero:at:infty}
\lim_{\substack{|X|\to\infty\\ X\in \overline{H_{m}}}} V^{(k)}_{ m}(X)=0.
\end{equation}
Moreover, from the definition of $\Phi$ (see~\eqref{eqn:potential}), it is clear that $v$ is bounded from above and below in $H_{m}$. More explicitly, setting $A_{m}=\sum\limits_{j\ge -m}e^{-q^{j}}$, we have 
\begin{equation}\label{eqn:potential:bounds}
e^{-\epsilon A_{m}/2} \le v (X)\le e^{\epsilon A_{m}/2},\quad X\in \overline{H_{m}}.
\end{equation}
Because of~\eqref{eqn:zero:at:infty}, there is $R>1$ such that $V^{(k)}_{m}(X)\le e^{-\epsilon A_{m}/2}(X)$ for any $X\in\overline{H_{m}}: |X-X_0|\ge R$. Therefore, by~\eqref{eqn:potential:bounds} and by the choice of $\eta_k$ we have $V^{(k)}_{ m} - v\le 0$ on $\partial \Bigl(B(X_0,R)\cap H_{m}\Bigr)$. Since, by~\eqref{eqn:Lv}, $L (V^{(k)}_{ m} - v)=-L v\le 0$, maximum principle, applied in the domain $B(X_0,R)\cap H_{m}$, gives 
\begin{equation}\label{eqn:step:main:lem}
V^{(k)}_{m}(X)\le v(X),\quad\forall X\in B(X_0,R)\cap H_m.
\end{equation}
Our aim is to improve this inequality at $X_0=(x_0,y_0,q^{-(m-1)})$ by using the maximum principle again. We look for a function $w\in Y^{1,2}_0(B(X_0,q^{-(m-1)}/4))$ with $w(X_0)=v (X_0)$ and $L (V^{(k)}_{ m} - v+bw)\le 0$ in $B(X_0,q^{-(m-1)}/4)$ for some $0<b<1$. Being this the case, the maximum principle, together with~\eqref{eqn:step:main:lem}, would prove~\eqref{eqn:exp:decr} with $\rho=(1-b)$ and $V^{(k)}_{ m}$ in place of $V_{ m}$. We set $w$ to be the polynomial
\[
w(X)=v(X_0)(1-16q^{2m-2}|X-X_0|^2),\quad X\in B(X_0,q^{-(m-1)}/4).
\]
Thanks to our previous discussion, we only need to check that there exists some $0<b<1$ such that $V^{(k)}_{ m} - v+bw$ is a subsolution in $B(X_0,q^{-(m-1)}/4)$. Since $V^{(k)}_{ m}$ is a solution, this amounts to checking the sign of $L (bw-v)$. By~\eqref{eqn:Lv} and~\eqref{eqn:size:field:1}, we have, for any $X\in B(X_0,q^{-(m-1)}/4)$ 
\begin{equation}\label{eqn:eqn:step:main:lem:1}
L v(X)\ge \frac{v(X)|\B|^2(X)}4\ge v(X)\frac{\epsilon^2}{400}q^{2m-2}.
\end{equation}
Moreover, using~\eqref{eqn:size:field}, for any such $X$ we have
\[
\frac{v(X)}{v(X_0)}=e^{-(\Phi(X)-\Phi(X_0))}\ge \inf\limits_{Y\in B(X_0,q^{-(m-1)}/4)}e^{-|\B(Y)||X-X_0|}\ge e^{-\frac{2\epsilon}{3}}.
\]
Together with~\eqref{eqn:eqn:step:main:lem:1}, this yields, for $X\in B(X_0,q^{-(m-1)}/4)$,
\begin{align*}
L (bw-v)(X)&\le bL w(X)- c_0q^{2m} v(X_0)\\
&=32bq^{-2}\Bigl(q^{2m}-\B(X)\cdot (X-X_0)\Bigr)v(X_0)-c_0q^{2m} v(X_0)\\
&\le v(X_0)\Bigl(32b q^{-2}(q^{2m}+2\epsilon/3)-c_0q^{2m}\Bigr),
\end{align*}
where we set $c_0=\frac{\epsilon^2}{400}q^{-2}e^{-\frac{2\epsilon}{3}}$.
Therefore, choosing $b=\frac{q^2c_0}{32}= \frac{\epsilon^2}{400}e^{-\frac{2\epsilon}{3}}$, we get 
\[
L (V^{(k)}_{ m}+bw-v)\le 0,
\]
that proves
\[
V^{(k)}_{ m}(X_0)\le \rho v(X_0).
\]
Sending $k\to\infty$ yields~\eqref{eqn:exp:decr}.

Assume now that $X_0=(x_0,y_0,z_0)$ with $z_0\ge 2$. For $0< j\le m$ we have, by Lemma~\ref{lem:comp:h},
\[
V_{j}(X_0)=\int_{\{z=q^{j-1}\}}V_{j}(x,y,q^{-(j-1)})\,d\omega^{X_0}_{j-1}.
\]
Applying~\eqref{eqn:exp:decr} at height $q^{-(j-1)}$ we have
\[
V_{j}(X_0)\le\rho \int_{\{z=q^{j-1}\}}v(x,y,q^{-(j-1)})\,d\omega^{X_0}_{j-1}=\rho V_{j-1}(X_0).
\]
An iteration of this inequality proves~\eqref{eqn:exp:decr:1}. Finally, ~\eqref{eqn:exp:decr:2} is a consequence of~\eqref{eqn:drop:base} and~\eqref{eqn:exp:decr:1}.
\end{proof}
We are now ready to prove~\eqref{eq:mainconclusion}.

\begin{lem}\label{lem:drop:concl}
Let $X_0=(x_0,y_0,z_0)$ be a point with $z_0\ge 2$. Let $\gamma=-\log(\rho)/(8\epsilon)$, where $\rho\in (0,1)$ is as in the statement of Lemma~\ref{lem:drop:key}. Then, for $m>>1$ be sufficiently large. 
Then,
\[
 \omega^{X_0}(\{S_m\le\gamma m\}\times\R)\le \rho^{\frac m 2}v(X_0)+C e^{-\alpha\frac \gamma 8 m},
\]
where $\alpha\in (0,1)$ is the boundary De Giorgi-Nash-Moser exponent as in Corollary~\ref{cor:carl}, and $C$ is a constant dependent on $\epsilon$ but not on $m$. In particular, 
\[
\omega^{X_0}(F)=1,
\]
where $F$ is as in Theorem~\ref{thm:dim:drop}.
\end{lem}
\begin{proof}
Let $m>>1$ and set $F^\gamma_m:=\{S_m\le\gamma m\}\times\R\times\{q^{-m}\}\subset\{z=q^{-m}\}$. By~\eqref{eqn:comp:h} we can decompose 
\begin{align}\label{eqn:dec}
 \omega^{X_0}(\{S_m\le\gamma m\}\times\R)&=\int_{\{z=q^{-m}\}}\omega^P(\{S_m\le\gamma m\}\times\R)\,d\omega_m(P)\\ \notag
&=\int_{F_{m}^{4\gamma}}\omega^P(\{S_m\le\gamma m\}\times\R)\,d\omega_m(P)\\\notag&+\int_{\{z=q^{-m}\}\setminus{F_{m}^{4\gamma}}}\omega^P(\{S_m\le\gamma m\}\times\R)\,d\omega_m(P)\\ \notag
&\le \omega_m(F_{m}^{4\gamma})+\int_{\{z=q^{-m}\}\setminus{F_{m}^{4\gamma}}}\omega^P(\{S_m\le\gamma m\}\times\R)\,d\omega_m(P).
\end{align}
Since $\gamma=-\log(\rho)/(8\epsilon)$, Lemma~\ref{lem:drop:key} yields
\begin{equation}\label{eqn:drop:fin}
\omega_m(F_{m}^{4\gamma})\le \rho^{\frac {m} 2}v(X_0).
\end{equation}
We now need to estimate the second term on the right-hand side of~\eqref{eqn:dec}.
For any $P\in \{z=q^{-m}\}\setminus F_{m}^{4\gamma}$, we have $P=(x_P,y_P,q^{-m})$with $S_m(x_P)>4\gamma m$. Hence, by~\eqref{eqn:size:sum} and for $m$ sufficiently large, 
\begin{equation}\label{eqn:drop:fin:1}
\Phi(P)\ge \epsilon4\gamma m - C\epsilon\ge \epsilon3\gamma m, \quad P\in \{z=q^{-m}\}\setminus F^{4\gamma}_m.
\end{equation}
We want to show that $\dist(P,\{S_m\le\gamma m\}\times\R)>>q^{-m}$. Let $(x',y')\in\R^2$ with $S_m(x')\le \gamma m$, and hence, setting $X'=(x',y',q^{-m})$, 
\begin{equation}\label{eqn:drop:fin:2}
\Phi(X')\le \epsilon\gamma m+C\epsilon\le 2\epsilon\gamma m,\quad X'\in F^\gamma_m.
\end{equation}
Therefore, if $|X'-P|\le q^{-m}$, by~\eqref{eqn:drop:fin:1}, ~\eqref{eqn:drop:fin:2}, and ~\eqref{eqn:size:field:1}, we have
\[
\epsilon\gamma m\le\Phi(P)-\Phi(X)\le 2\epsilon,
\]
that is a contradiction. Hence, $|X'-P|> q^{-m}$, and so, using~\eqref{eqn:size:horizontal} in place of~\eqref{eqn:size:field:1}, we have
\[
\epsilon\gamma m\le\Phi(P)-\Phi(X)\le 2\epsilon+4\epsilon\log (q^m|X'-P|).
\]
Thus, setting $\pi(P):=(x_P,y_P)$, if $m$ is sufficiently large we get
\begin{equation}\label{eqn:drop:fin:3}
\dist(P,\{S_m\le\gamma m\}\times\R)\ge\dist(\pi(P),\{S_m\le\gamma m\}\times\R) > e^{\frac \gamma 8m}q^{-m}.
\end{equation}
Set $R_m=e^{\frac \gamma 8 m}q^{-m}$, and, for $i\in\N$, take $\le f_i\in C^\infty_c(\R^2\setminus \overline{B(0,R_m)})$ with $0\le f_i\le 1$ and $f_i=1$ in $B(0,2^iR_m)\setminus B(0,(1+2^{-i})R_m)$. By~\eqref{eqn:hold:c} we have, for $m$ sufficiently large, 
\[
\int_{\R^2}f_i\,d\omega^P\le C\left(\frac{q^{-m}}{R_m}\right)^\alpha\le C e^{-\alpha\frac \gamma 8 m}.
\] 
Taking the limit as $i\to\infty$ and using~\eqref{eqn:drop:fin:3} we get
\begin{equation}\label{eqn:drop:fin:4}
\omega^P(\{S_m\le\gamma m\}\times\R)\le\omega^P(\R^2\setminus \overline{B(0,R_m)})\le C e^{-\alpha\frac \gamma 8 m},
\end{equation}
where $C$ is a constant dependent on $\epsilon$, but uniform in $m$ and in $P\in \{z=q^{-m}\}\setminus F_{m}^{4\gamma}$. 
Therefore, plugging~\eqref{eqn:drop:fin:4} and ~\eqref{eqn:drop:fin} into ~\eqref{eqn:dec}, we obtain
\begin{equation}\label{eqn:drop:fin:5}
\omega^{X_0}(\{S_m\le\gamma m\}\times\R)\le \rho^{\frac m 2}v(X_0)+C e^{-\alpha\frac \gamma 8 m}.
\end{equation}
Notice that 
\[
\R^2\setminus F=\left\{\liminf\limits_{m\to\infty}\frac{S_m}{m}< \gamma\right\}\times\R=\bigcup\limits_{k\in\N}\bigcap\limits_{j\in\N}\bigcup\limits_{m\ge j}\{S_m\le \gamma m(1-2^{-k})\}\times\R.
\]
By~\eqref{eqn:drop:fin:5} we have 
\[
\omega^{X_0}\left(\bigcap\limits_{j\in\N}\bigcup\limits_{m\ge j}\{S_m\le \gamma m\}\times\R\right)\le\lim_{j\to\infty}\sum_{m\ge j}\rho^{\frac m 2}v(X_0)+C e^{-\alpha\frac \gamma 8 m}=0.
\]
Therefore, $\omega^{X_0}(\R^2\setminus F)=0$ as well.
\end{proof}

To complete the proof of Theorem~\ref{thm:dim:drop}, we may now prove that the Hausdorff dimension of $F$ is strictly smaller than $2$. For a proof of the next Lemma on lacunary sums, see~\cite[Lemma 1]{Azu67}.
\begin{lem}\label{lem:deviation}
For every $t>0$ and integer $m\ge0$,
\begin{equation}\label{eq:exponential}
 \frac1{2\pi}\int_0^{2\pi}e^{tS_m(x)}\, dx\le e^{mt^2/2},
 \qquad S_0=0.
\end{equation}
Consequently, for every $t>0$,
\begin{equation}\label{eq:largedeviation}
 |\{x\in[0,2\pi):S_m(x)\ge tm\}|
       \le2\pi e^{-t^2m/2}.
\end{equation}
\end{lem}
We perform now a simple covering argument to end the proof of the dimension drop.
\begin{lem}
Let $q$, $F$, and $E$ be as in Theorem~\ref{thm:dim:drop} and let $\gamma$ be as in Lemma~\ref{lem:drop:concl}. Then,
\begin{equation}\label{eq:dimension}
\dimH E<1-\frac{\gamma^2}{32\log q}.
\end{equation}
In particular,
\[
\dimH F<2-\frac{\gamma^2}{32\log q}
\]
\end{lem}
\begin{proof}
For $m\ge 1$, partition $[0,2\pi)$ into the $q^m$ closed intervals $I_{i}=[\frac{2\pi i}{q^m},\frac{2\pi(i+1)}{q^m})$, $i=0,\dots,q^{m}-1$, having length
$\ell_m=2\pi q^{-m}$. Since
\[
 |S_m'(x)|\le\sum_{j=0}^{m-1}q^j
              =\frac{q^m-1}{q-1},
\]
the oscillation of $S_m$ on any such interval is at most
\begin{equation}\label{eq:oscillationconstant}
 D_0=\frac{2\pi}{q-1}.
\end{equation}
Therefore, for $m\ge4D_0/\gamma$ and $0\le i\le q^{m-1}$, if
$I_i\cap\{S_m\ge\gamma m/2\}\ne\emptyset$, every point $x\in I_i$ satisfies
$S_m(x)\ge\frac{\gamma m}{2}-D_0\ge\gamma m/4$.
Let $N_m$ be the number of such intervals $I_i$. Applying
\eqref{eq:largedeviation} with $t=\gamma/4$ gives
\[
 N_m\ell_m\le \left| \{S_m\ge\frac\gamma 4 m\}\cap[0,2\pi)\right|\le  2\pi e^{-\frac{\gamma^2m}{32}},
\]
and hence 
\begin{equation}\label{eqn:cover}
 N_m\le q^me^{-\frac{\gamma^2m}{32}}.
\end{equation}
For each integer $N\ge1$, set
\[
 E_N=\bigcap_{m\ge N}\left\{S_m\ge\frac{\gamma m}{2}\right\}\cap[0,2\pi).
\]
The definition of $E$ implies
$E\cap[0,2\pi)\subset\bigcup_NE_N$.
For any fixed $N$ and arbitrarily large $m\ge N$, the $N_m$
selected partition intervals cover $E_N$.

To prove $\mathcal{H}^s(E)=0$ it suffices to find covers of
$E$ with diameters tending uniformly to zero and sums of the $s$th
powers of those diameters tending to zero. By~\eqref{eqn:cover}, our covers
satisfy
\[
 N_m\ell_m^s
 \le q^{m(1-s)}e^{-\frac{\gamma^2m}{32}}(2\pi)^s=(2\pi)^s\exp\left(m\left[(1-s)\log q
                                      -\frac{\gamma^2}{32}\right]\right)
 \longrightarrow0,
\]
whenever $s>1-\frac{\gamma^2}{32\log q}$.
Thus, for such $s$, $\mathcal{H}^s(E_N)=0$ for every $N$, that implies $\mathcal{H}^s(E\cap[0,2\pi])=0$.
All the functions $S_m$ have period $2\pi$, so $E$ has this period as well. In particular, $\mathcal{H}^s(E)=0$ for
every $s>1-\frac{\gamma^2}{32\log q}$, which is \eqref{eq:dimension}.
\end{proof}
\begin{rmk}
Notice that, thanks to~\cite{HofL01}, the vector field $\B$ constructed in Theorem~\ref{thm:dim:drop} cannot satisfy the Carleson measure condition~\eqref{eqn:carl:cond}. This can be easily seen by a direct computation. Indeed, let $B$ be the ball of radius $1$ centered at $0$. Notice that, by~\eqref{eqn:size:field:1}, we have 
\[
\int_{B\cap\R^3_+}|\B|^2(z)\,z\,dz\ge C^{-1}\sum_{m\ge 10}q^{2m}\int_{q^{-m}/2}^{2q^{-m}}z\,dz=\infty.
\]
\end{rmk}
\begin{rmk}
Let $\Phi$ be as in~\ref{thm:dim:drop}, and let $0<\epsilon<\frac{1}{2}$. We claim that $w=v^2=e^{-\Phi}$ is an $A_2$-weight. We need to check that
\begin{equation}\label{eqn:A2}
\sup_{Q}\fint_{Q\cap\R^3_+} w\fint_{Q\cap\R^3_+}w^{-1}=\sup_Q\fint_{Q\cap\R^3_+} \fint_{Q\cap\R^3_+}e^{\Phi(X)-\Phi(Y)}\,dX\,dY\le C<\infty,
\end{equation}
where the supremum is taken over cubes centered in $\overline{\R^3_+}$. Let $Q$ be such a cube with side length $\ell$. Assume  that $\ell\ge \dist(Q,\R^2)$. Then, for any $X, Y\in Q$, by~\eqref{eqn:size:field} we have 
\[
 e^{|\Phi(X)-\Phi(Y)|}\le e^{2\sqrt{3}\epsilon},
\]
proving~\eqref{eqn:A2} for such cubes $Q$. Assume instead that $\dist(Q,\R^2)<\ell$. Let $Q=I\times J$, where $I$ is a square and $J$ is an interval. Let $X_I$ be the center of $I$ and let $X\in Q$. Let $X_Q=(X_I,2\ell)$ We estimate from above $\Phi(X)-\Phi(X_Q)$ using the same argument as in the proof of~\eqref{eqn:size:horizontal}. That is, let $\Bar X$ be the point having height $2\ell$ and the same $(x,y) $- coordinate as $X$. Going vertically from $X$ to $\Bar X$ and then horizontally to $X_Q$, using~\eqref{eqn:size:field} we get
\[
|\Phi(X)-\Phi(X_Q)|\le 2\epsilon\int_{\delta(X)}^{2\ell}\frac{1}{z}\,dz+\sqrt{2}\epsilon=2\epsilon\log (2\ell/\delta(X))+\sqrt{2}\epsilon.
\]
Therefore, taking the exponential, writing $X=(x,y,z)$, using that $\epsilon<\frac{1}{2}$ and integrating we get
\[
\fint_Q \fint_Qe^{\Phi(X-\Phi(Y))}\,dX\,dY\le C+ \ell^{-2} \left(\int_{0}^{2\ell}\left(\frac{2\ell}{z}\right)^{2\epsilon}\,dz\right)^2\le C_\epsilon,
\]
completing the proof of~\eqref{eqn:A2}. Therefore, if we rewrite the operator $L=-\Delta+\B\cdot\nabla$ as 
\[
L=-w^{-1}\Div(w\nabla),
\]
we find that $L$ is a divergence form operator weighted with an $A_2$-weight. Theory of such operators was developed in~\cite{FabJK82B}. From the dimension drop of Theorem~\ref{thm:dim:drop}, we can therefore deduce a dimension drop for the harmonic measure of a divergence form operator weighted with an $A_2$-weight.
\end{rmk}

\section*{Disclosure of AI assistance}

The construction of the drift $\B$ and the proof of
Theorem~\ref{thm:dim:drop} were generated by ChatGPT 6-Astra
through a series of exchanges with the authors. The authors
initially suggested a preliminary idea concerning the qualitative
behavior of a drift that might produce a dimension drop, and
subsequently requested clarifications and revisions of the argument.
The authors rewrote the resulting presentation for clarity and
readability and take full responsibility for the mathematical
content of the paper.

\def\cprime{'}
\bibliographystyle{amsalpha}
\bibliography{bibli}
\end{document}